\documentclass[12pt,a4paper]{article}

\newif\ifrelease
\releasetrue

\ifrelease
	\newcommand{\MARK}[1]{}
\else
	\newcommand{\MARK}[1]{{\color{red}#1}}
	\usepackage[notref,notcite]{showkeys}
	\renewcommand{\showkeyslabelformat}[1]%
   		{{\!\!{\color{blue}\tiny\sffamily#1}\!\!}}
\fi 

\newcommand{\drop}[1]{}

\usepackage[margin=25.4mm]{geometry} 
\usepackage[utf8]{inputenc}  
\usepackage[T1]{fontenc} 
\usepackage{xcolor,scrtime,graphicx,pinlabel}
\usepackage{amsmath,amsfonts,relsize}
\usepackage{mathrsfs}
\usepackage{amsthm}
\usepackage{booktabs}
\usepackage{tikz}
	\usetikzlibrary{decorations.pathmorphing,calc}
\usepackage{float}   
\usepackage{rotating}

\newcommand{\pdelta}{\bm{\updelta}}  

\usepackage{enumitem}
  \setlist{noitemsep}

\usepackage{dsfont}
\newcommand{\bONE}{\mathds{1}}

\newtheorem*{answer}{Basic Answer}

\newtheorem{theorem}{Theorem}[section]
\newtheorem{lemma}[theorem]{Lemma}
\newtheorem{definition}[theorem]{Definition}

\newtheorem{remark}[theorem]{Remark}

\newtheorem{Ftheorem}[theorem]{Formal Theorem}

\theoremstyle{definition}

\usepackage{thmtools}
\declaretheoremstyle[notefont=\bfseries,
spacebelow=\bigskipamount]{boldnote}
\declaretheoremstyle[%
spacebelow=\bigskipamount]{normalnote}
\declaretheoremstyle[spacebelow=\bigskipamount]{default}

\declaretheorem[style=normalnote,name=Definition,numberlike=theorem,qed={\lower-0.3ex\hbox{$\square$}}]{definition}

\declaretheorem[style=normalnote,name=Remark,numberlike=theorem,qed={\lower-0.3ex\hbox{$\square$}}]{remark}

\usepackage{textcomp}
\theoremstyle{definition}

\usepackage{mdframed}
\theoremstyle{definition}
\newtheorem{examplex}{Example}

\newenvironment{Ex}
{\pushQED{\qed}%
	\begin{mdframed}[
	innertopmargin=-1pt,
	linewidth=1pt,linecolor=gray!50]\examplex}%
{\popQED\endexamplex\end{mdframed}}

\newcommand{\ti}{{\times}}

\usepackage[rm={oldstyle=false}]{cfr-lm} 
\usepackage{fixcmex} 
\usepackage{bold-extra} 
\usepackage{xspace}
\newcommand\generic{\textsc{Generic}\xspace}
\let\Generic\generic

\numberwithin{equation}{section}
\usepackage{bm} 

\usepackage{upgreek}

\let\ds\displaystyle
\let\wh\widehat
\def\R{{\mathbb R}} \def\N{{\mathbb N}}

\def\dd{\,\mathrm{d}} 
\def\ee{\mathrm{e}}     

\newcommand{\Lebesgue}{\mathscr L}
\newcommand{\RelEnt}{\calH}
\DeclareMathOperator\Prob{Prob}
\newcommand{\e}{\varepsilon}
\renewcommand\div{\mathop{\mathrm{div}}\nolimits}
\newcommand{\ProbMeas}{\calP}
\newcommand{\Expectation}{\bbE}
\renewcommand{\Re}{\mathop{\mathrm{Re}}}
\DeclareMathOperator{\supp}{supp}
\newcommand{\wt}{\widetilde}
\DeclareMathOperator{\dom}{dom}
\newcommand{\Lip}{\mathop{\mathrm{Lip}}}
\newcommand{\Lipc}{\mathop{\mathrm{Lip}_{\mathrm c}}}

\DeclareFontFamily{U}{mathx}{}
\DeclareFontShape{U}{mathx}{m}{n}{<-> mathx10}{}
\DeclareSymbolFont{mathx}{U}{mathx}{m}{n}
\DeclareMathAccent{\widecheck}{0}{mathx}{"71}

\usepackage{mathtools} 
\DeclarePairedDelimiter{\abs}{\lvert}{\rvert}

\DeclarePairedDelimiter{\bra}{(}{)}
\DeclarePairedDelimiter{\pra}{[}{]}
\let\set\undefined
\DeclarePairedDelimiter{\set}{\{}{\}}

\usepackage{ifthen}
\newlength{\leftstackrelawd}
\newlength{\leftstackrelbwd}
\def\leftstackrel#1#2{\settowidth{\leftstackrelawd}%
{${{}^{#1}}$}\settowidth{\leftstackrelbwd}{$#2$}%
\addtolength{\leftstackrelawd}{-\leftstackrelbwd}%
\leavevmode\ifthenelse{\lengthtest{\leftstackrelawd>0pt}}%
{\kern-.5\leftstackrelawd}{}\mathrel{\mathop{#2}\limits^{#1}}}

 \def\calE{{\mathcal E}} \def\calF{{\mathcal F}}
 \def\calH{{\mathcal H}}

\def\calP{{\mathcal P}}  \def\calR{{\mathcal R}}
\def\calS{{\mathcal S}}  
  \def\calX{{\mathcal X}}
\def\calY{{\mathcal Y}} 
  
\def\rmd{{\mathrm d}}

\def\rmD{{\mathrm D}}

\def\sfp{{\mathsf p}} \def\sfq{{\mathsf q}}

 \def\sfZ{{\mathsf Z}}

 \def\bbE{{\mathbb E}}

\def\bbS{{\mathbb S}}

  \def\bfX{{\mathbf X}} 
 \def\bfZ{{\mathbf Z}}

\def\scrJ{{\mathscr  J}}

\usepackage[
	style=alphabetic,
	natbib=false,
	url=false,
	doi=false,
	isbn=false,
	giveninits=true,
	maxbibnames=99,
	backend=biber]{biblatex}
\renewbibmacro{in:}{} 
\DeclareFieldFormat[article]{title}{{#1}}
\usepackage[colorlinks=true, linkcolor=black, citecolor=black, urlcolor=black,hypertexnames=false]{hyperref}

\begin{document}

\title{Why does entropy drive evolution equations?}

\author{Mark A. Peletier}


\maketitle

\begin{abstract}
	`Entropy' appears as driving force in many different evolution equations, both deterministic and stochastic, and in these equations this `entropy' also takes different forms. We show how all these examples can  be understood as different instances of a common principle: Entropy drives evolutions because it characterizes the invariant measure of an underlying stochastic process. This interpretation explains the appearance of entropy, the different forms that entropy takes in these equations, and how entropy `drives' these evolution equations.
	We illustrate this common structure with examples from stochastic processes, gradient flows, and \Generic systems. 
	

\end{abstract}

\section{Introduction}
\label{se:central_question}

\subsection{The basic question}
\label{ss:basic-question}
This paper starts from the observation that many evolution equations are `driven' by something that people call `entropy'---we give examples below. The aim of this paper is to discuss exactly what this means:  what it means to be called an `entropy',  and why this object `drives' some evolution.

\medskip

We first give a mathematical definition of `driving an evolution'. In this paper we use the \textsl{metriplectic} or \textsl{\Generic} framework, which is an axiomatic description of a class of evolution equations that includes both Hamiltonian systems and gradient flows. We succinctly describe this framework here, with details relegated to Appendix~\ref{s:generic}, and we illustrate the framework in the examples below.

Given a state space $\bfZ$ with states $z\in \bfZ$, metriplectic or \Generic evolution equations are of the form
\begin{equation}
	\label{eq:GENERIC}
	\dot z = J(z) \rmD \calE(z) + K(z) \rmD \calS(z).
\end{equation}
Here $\calE,\calS:\bfZ\to \R$ are functionals, $\rmD$ indicates differentiation, and for each $z$, $J(z)$ and~$K(z)$ are linear operators satisfying (using matrix notation)
\begin{equation}
	\label{eq:JK-symmetry-intro}
	J(z)^\top = -J(z)\qquad  \text{and} \qquad K(z)^\top = K(z)\geq0.
\end{equation}
The operator $J$ is commonly called the \emph{Poisson} operator, and $K$ the \emph{Onsager} operator. 
In addition, the following \emph{non-interaction conditions} are required to hold:
\begin{equation}
	\label{eq:NIC-intro}
	J(z)\rmD \calS(z) = 0 \qquad \text{and} \qquad K(z)\rmD \calE(z) = 0.
\end{equation}

In the \Generic and metriplectic literature the functionals $\calE$ and $\calS$ are generally called `energy' and `entropy'. The symmetry and non-interaction properties of $J$ and $K$ imply that along the evolution~\eqref{eq:GENERIC} the functional~$\calE$ is conserved, and $\calS$ increases:
\begin{align*}
	\frac{\dd}{\dd t} \calE(z(t)) = \rmD \calE \cdot J\rmD \calE + \rmD \calE \cdot K\rmD\calS \ \stackrel{(\ref{eq:JK-symmetry-intro},\ref{eq:NIC-intro})}  = \ 0,\\
	\frac{\dd}{\dd t} \calS(z(t)) = \rmD \calS \cdot J\rmD \calE + \rmD \calS \cdot K\rmD\calS \ \stackrel{(\ref{eq:JK-symmetry-intro},\ref{eq:NIC-intro})}\geq \ 0.
\end{align*}
A special case is that of a \emph{gradient flow}\footnote{Note that gradient flows come with different signs; we follow the Physics convention, in which an entropy tends to increase.}, which is a \generic system in which $\calE\equiv 0$:
\begin{equation}
	\label{eq:GF}
	\dot z =  K(z) \rmD \calS(z).
\end{equation}
In the context of this paper, therefore, we say that in an equation of the form~\eqref{eq:GENERIC} the functionals $\calE$ and $\calS$ \emph{drive} the evolution, and in particular we focus on the  driving by~$\calS$.

\bigskip
We now return to the main question of this paper, 
\begin{quote}
	\emph{Q0. Why does `entropy' appear as driver in many evolution equations?}
\end{quote}
or more precisely, 
\begin{quote}
\emph{Q0'. Why are there many evolution equations of the form~\eqref{eq:GENERIC} or~\eqref{eq:GF} in which people call $\calS$ an `entropy'?}
\end{quote}

Part of the difficulty here is that `entropy' tends to be a confusing concept. There is no single definition, only many different instances, in contrast to e.g.\ a vector space or a measurable function. It also does not help that things called `entropy' appear in many different forms and roles in the literature. 
Nonetheless, `entropy' is often \emph{said} to drive evolution equations.

\newcounter{entropynumber}
\renewcommand\theentropynumber{(\roman{entropynumber})}
\newcommand\num{\refstepcounter{entropynumber}\theentropynumber}
\bgroup
\def\arraystretch{2.2}%
\begin{sidewaystable}[p]
	\small
	\centering
\begin{tabular}{rllll}
\toprule
&Formula & Context & Coarse-graining map & See \\
\midrule
\noalign{\vskip2\jot}
\num&$\calS(Q,P,e) = \beta e$ & Damped 	 oscillator & $ (Z,z) \mapsto \bra[\big]{Z\,,\, e = \tfrac12 \|z - C_n Z\|_{\calH_n}^2 - \tfrac n\beta}$  & Ex.~\ref{ex:damped-harmonic-oscillator}, Sec.~\ref{ss:CG-CGinvmeas-entropy-ExA}\\[\jot]
\num&$\calS(Q,P,\theta) = \log \theta$ & Damped 	 oscillator &$(*)$& Ex.~\ref{ex:damped-harmonic-oscillator}, \cite{JungelStefanelliTrussardi21}\\
\num\label{ent:diffusion}&$\ds\calS(\rho) = -\int\rho(x)\log \rho(x)\dd x$  & Diffusion &$(x_1,\dots,x_n) \mapsto \tfrac1n \sum_i \delta_{x_i}$ & Ex.~\ref{ex:JKO}, Sec~\ref{ss:relative-entropies-Sanov}\\
\num\label{ent:FP}&$\ds\calS(\rho) = -\int\rho(x)\bra[\big]{\log \rho(x) + V(x)}\dd x$  & Fokker-Planck equation &$(x_1,\dots,x_n) \mapsto \tfrac1n \sum_i \delta_{x_i}$ & Ex.~\ref{ex:JKO}, Sec~\ref{ss:relative-entropies-Sanov}\\
\num\label{ent:ZRP}&$\ds\calS(\rho) = -\int s(\rho(x))\dd x$, \ $s$ non-explicit  & Zero-range process &$(x_1,\dots,x_n) \mapsto \tfrac1n \sum_i x_i \delta_{i/n}$& Ex.~\ref{ex:DSZ}, Sec.~\ref{ss:interacting-particles}\\
\num\label{ent:hard-rods}&$\ds\calS(\rho) = -\int \rho(x) \log \frac{\rho(x)}{1-a\rho(x)}\dd x$  &  {Hard-rod} system &$(x_1,\dots,x_n) \mapsto \tfrac1n \sum_i \delta_{x_i}$& Ex.~\ref{ex:DSZ}, Sec.~\ref{ss:interacting-particles}\\
\num\label{ent:heat-conduction}&$\ds\calS(\rho) = \int \log \rho(x)\dd x$  & Heat conduction &$(x_1,\dots,x_n) \mapsto \tfrac1n \sum_i x_i \delta_{i/n}$& Ex.~\ref{ex:DSZ}, Sec.~\ref{ss:interacting-particles}\\
\num&$\ds\calS(y) =  \log \int \pdelta[\Phi(x)-y](\rmd x)$ & Coarse-graining &$x\mapsto \Phi(x)$ & Sec.~\ref{s:core-example}, \eqref{eqdef:CG-example-S}\\[2\jot]
\bottomrule
\end{tabular}
\caption{A number of different formulae for entropies that drive evolution equations. See the referenced sections for details. The entry $(*)$ in the second row indicates that there is no explicit construction yet of the microscopic system and the coarse-graining map in the literature.}
\label{table:different-entropies}
\end{sidewaystable}
\egroup

\bigskip
The aim of this paper is to formulate an answer to Question Q0 above, and we will naturally touch on various  related questions such as
\begin{enumerate}[label=Q\arabic*]
\item \label{q:what-is-entropy-from-modelling-perspective}
What does it mean to be called an `entropy'? 
How should one interpret this from a modelling perspective?
\item \label{q:why-entropy-in-so-many-forms}
Why does `entropy' come in so many different functional forms (see Table~\ref{table:different-entropies})?
\item\label{subq:increasing} Does `entropy' always increase along an evolution? 
\item \label{q:modelling-interpretation-K} Given that the operator $K$ characterizes \emph{how} $\calS$ drives the evolution in~\eqref{eq:GENERIC}, what is the modelling interpretation of $K$?
\item\label{q:interp-E-J} In the same vein, what is the modelling interpretation of $\calE$ and $J$?
\item \label{q:unique-char-of-GENERIC} 
Does the equation uniquely characterize $\calE$, $\calS$, $J$, and $K$?
\end{enumerate}
In the course of this paper we encounter various parts of answers to all these questions; in Section~\ref{s:conclusion} we  summarize the understanding that we have gained.

\subsection{Example: Damped oscillator}
\label{ss:damped-oscillator-intro}

We now discuss a first example of a system that is driven by an entropy in this sense.

\medskip

\begin{Ex}[The damped harmonic oscillator as a \generic system]

\label{ex:damped-harmonic-oscillator}

The sim\-plest example of a \generic system~\eqref{eq:GENERIC} is the damped harmonic oscillator~\cite{Ottinger18}. Here the state variable is $(Q,P,e)\in \R^3$, which describes the position $Q$ and momentum~$P$ of an object of mass $m$ connected to a spring with constant $k$ and a damper with constant~$\gamma$; the variable $e$ has the interpretation of an energy, as we discuss below. 

\begin{center}
\begin{tikzpicture}[scale=1.2,line cap=round]

\draw[thick] (0,0.2) -- (0,1.8);
\foreach \y in {0.2, 0.4, 0.6,...,1.8} {
  \draw[thick] (-0.2,\y) -- (0,\y+0.2);
}

\draw[thick] (0,1) -- (1,1);
\draw[thick] (1,0.5) -- (1,1.5);

\begin{scope}[xshift=1cm]
\draw[thick] (0,1.5) -- (0.3,1.5);
\draw[thick,decorate,decoration={aspect=0.4, segment length=3mm, amplitude=2mm,coil}] 
  (0.3,1.5) -- (1.7,1.5) node[midway, above, shift={(0,0.25)}] {$k$};
  \draw[thick] (1.7,1.5) -- (2,1.5) coordinate (springR);

\draw[thick] (0,0.5) -- (0.7,0.5); 

\draw[thick] (0.7,0.3) -- (1.3,0.3) node[below, midway, shift={(0,-0.05)}] {$\gamma$};     
\draw[thick] (0.7,0.7) -- (1.3,0.7);     
\draw[thick] (0.7,0.3) -- (0.7,0.7);     

\draw[thick] (1.0,0.32) -- (1.0,0.68);   
\draw[thick] (1.0,0.5) -- (2,0.5);       
\draw[thick] (1.0,0.5) -- (2,0.5) coordinate (A); 
\draw[thick] (1.0,0.3) -- (1.0,0.7);

\draw[thick] (springR) -- (A);
\draw[thick] ($ (A)!0.5!(springR) $) -- (3,1);

\begin{scope}[xshift=1cm]
\draw[thick,fill=gray!20] (2,0.7) rectangle (3,1.3);
\node at (2.5,1) {$m$};

\draw[->,thick] (3.2,1) -- (3.9,1)node[midway,above] {$Q$} ;
\end{scope}

\end{scope}
\end{tikzpicture}
\end{center}
The variables $Q$, $P$, and $e$  satisfy the evolution equations
\begin{equation}
	\label{eq:ODE-harmonic-oscillator}
	\dot Q = \frac Pm , \qquad 
	\dot P = - kQ - \gamma \frac Pm,\qquad \text{and}\qquad
	\dot e = \gamma \frac{P^2}{m^2}.
\end{equation}
We recognize the first equation as the relationship between velocity $\dot Q$ and momentum~$P$, and the second equation as the force balance connecting the momentum change $\dot P$ to the total force $-kQ -\gamma P/m$. This force consists of  a conservative restoring force $-kQ$, corresponding to the elastic energy $kQ^2/2$ of a spring, and a dissipative force $-\gamma P/m$ corresponding to linear friction with parameter $\gamma>0$. 

The friction term $-\gamma P/m$ causes the total mechanical energy to decay along a solution of~\eqref{eq:ODE-harmonic-oscillator}:
\[
\frac{\rmd}{\rmd t} \bra*{\frac{P(t)^2}{2m} + \frac k2 Q(t)^2 } 
= \frac Pm \bra*{- kQ - \gamma \frac Pm} + kQ \, \frac Pm 
= -\gamma \frac{P^2}{m^2}\leq0.
\]
This observation allows us to interpret $e$, with equation $\dot e = \gamma P^2/m^2$,  as the total amount of mechanical energy that has been `lost' due to friction since the start of the evolution. Put differently, if we define a `total energy'
\[
\calE(Q,P,e) := \frac{P^2}{2m} + \frac k2 Q^2 + e, 
\]
then this total energy is conserved:
\[
\frac{\dd}{\dd t} \calE\bra[\big]{Q(t),P(t),e(t)}
=  0.
\]

We can now write equation~\eqref{eq:ODE-harmonic-oscillator} in the form of a \generic equation~\eqref{eq:GENERIC} as follows:
\begin{equation}
	\label{eq:Generic-ODE-damped-harmonic-oscillator}
\frac{\dd}{\dd t} \begin{pmatrix}
	Q \\ P \\e
\end{pmatrix}
= 
\underbrace{\begin{pmatrix}
	0 & 1 & 0 \\ 
	-1& 0 & 0 \\
	0 & 0 & 0
\end{pmatrix}}_{J}
\underbrace{\begin{pmatrix}
	kQ \\ P/m \\ 1
\end{pmatrix}}_{\rmD\calE}
+
\underbrace{\frac\gamma\beta \begin{pmatrix}
	0 &    0 & 0\\
	0 &    1 & -P/m\\
	0 & -P/m & P^2/m^2\\
\end{pmatrix}}_{K}
\underbrace{\begin{pmatrix}
	0 \\ 0 \\ \beta 
\end{pmatrix}}_{\rmD \calS}.
\end{equation}
Here $\beta>0$ is a parameter that we discuss in Section~\ref{s:damped-pendulum}, and we have made a choice for the `entropy' $\calS$,
\[
\calS(Q,P,e) := \beta e.
\]
These choices for $\calE$, $\calS$, $J$, and $K$ indeed satisfy the requirements~\eqref{eq:JK-symmetry-intro} and~\eqref{eq:NIC-intro}. As already remarked above we therefore find that $\calS$ increases along an evolution:
\[
\frac{\dd}{\dd t} \calS\bra[\big]{Q(t),P(t),e(t)} = \beta  \dot e 
= \beta \gamma \frac{P^2}{m^2}\geq  0.\qedhere
\]
This behaviour can also be recognized in a numerical simulation:
\begin{figure}[H]
\begin{center}
	\labellist
	\pinlabel {\footnotesize $e$} [b] at 541 173 
	\pinlabel {\footnotesize $Q$} [tr] at 52 220
	\pinlabel {\footnotesize $P$} [r] at 58 43
	\pinlabel {\footnotesize $t$} [t] at 509 13
	\endlabellist
\includegraphics[width=0.9\hsize]{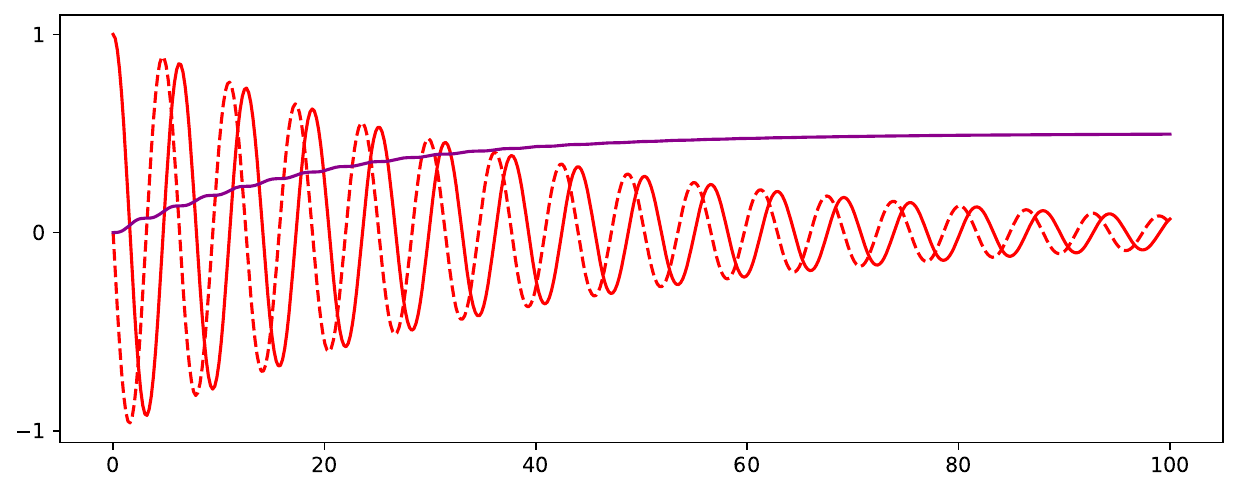}
\end{center}
\caption{Simulation of equation~\eqref{eq:ODE-harmonic-oscillator}. Note how $e$ is monotonically increasing, and the two other variables converge to zero. As $t\to\infty$, all energy is transferred from the mechanical components $P^2/2m+kQ^2/2$ to $e$, while keeping the sum $\calE$ constant. In this simulation $m=1$, $k=1$, $\gamma=0.05$, and $\beta=1$.}
\end{figure}
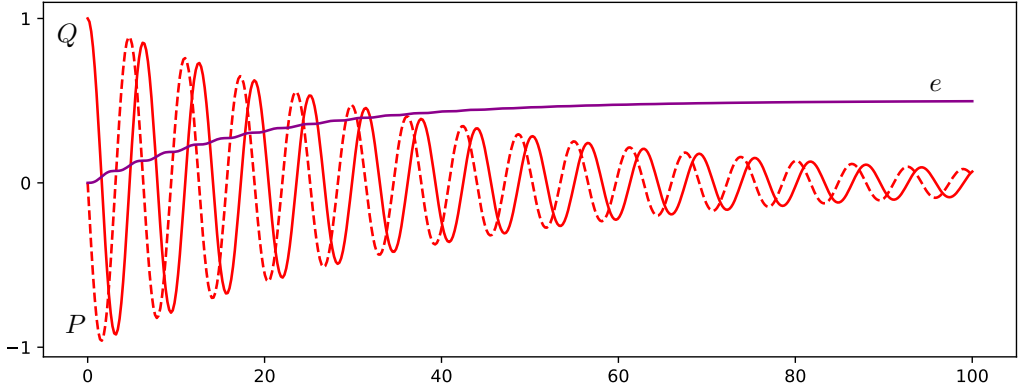

\medskip
See~\cite{JungelStefanelliTrussardi21} for a version of the damped oscillator with varying temperature~$\theta$, and where entropy has the functional form $\calS(Q,P,\theta) = \log \theta$.
\end{Ex}

\bigskip
\noindent 
This example illustrates the basic question of this paper. The choice $\calS = \beta e$ allows us to write the equation~\eqref{eq:ODE-harmonic-oscillator} in the \generic form~\eqref{eq:GENERIC}. Comparing these two  equations shows us that $\calS$ generates the friction terms parametrized by $\gamma$ in~\eqref{eq:ODE-harmonic-oscillator}. This leads to more questions of the type we started with, such as: Why is \emph{friction} driven by something called `entropy'? Is the lost energy $e$ stored somewhere? Why is $\calS(Q,P,e) = \beta e$ the `right' formula?  What is $\beta$? Why is $\calS$ a function of this lost energy $e$?
We will return to all these questions at the end of Section~\ref{s:damped-pendulum}.

\subsection{Other examples}
\label{ss:intro-other-examples}

The question Q0, `what does it mean to be an entropy?' becomes even more interesting from noting that there are many other systems that are similarly driven by something called `entropy', and that these entropies have a range of different functional forms. 

\begin{Ex}[Diffusion as gradient flow of entropy]
\label{ex:JKO}
Jordan, Kinderlehrer, and Otto \cite{JordanKinderlehrerOtto97,JordanKinderlehrerOtto98,Otto01} show that the Fokker-Planck equation
\begin{equation}
	\label{eq:FPeq}
	\partial_t \rho = \Delta \rho + \div \pra[\big]{\rho \nabla V} \qquad\text{in } \R^d
\end{equation}
can be written as $\dot \rho = K(\rho) \rmD \calS(\rho)$ (i.e.~\eqref{eq:GENERIC} with $\calE=0$) for 
\begin{equation}
	\label{eqdef:S-K-JKO}
	\calS(\rho) := -\int_{\R^d} \bra[\big]{\rho(x)\log \rho(x) + \rho(x) V(x) }\, \dd x ,
	\qquad K(\rho)\xi  := -\div \pra[\big]{\rho \nabla \xi} \text{ for }\xi:\R^d\to\R.
\end{equation}
Therefore the evolution equation~\eqref{eq:FPeq} can be considered to be driven by $\calS$ in~\eqref{eqdef:S-K-JKO}. The first term in  $\calS$ is commonly called the \emph{Gibbs-Boltzmann entropy} of $\rho$ (e.g.~\cite{JordanKinderlehrerOtto98}).

Since the publication of~\cite{JordanKinderlehrerOtto98} and~\cite{Otto01} this setup has been generalized in many different ways, e.g.\ including interaction terms~\cite{CarrilloMcCannVillani03}, higher-order diffusion~\cite{GianazzaSavareToscani09,MatthesRott22}, nonlinear diffusion (see Example~\ref{ex:DSZ} next), and others. 
\end{Ex}

\begin{Ex}[Nonlinear diffusion as gradient flow of various entropies]
\label{ex:DSZ}
Otto~\cite{Otto01} shows that for certain $\phi:[0,\infty)\to[0,\infty)$ the nonlinear diffusion equation for a density $\rho = \rho(t,x):[0,\infty)\times \R^d \to [0,\infty)$,
\begin{equation}
	\label{eq:nonlinear-diffusion-PDE}
	\partial_t \rho = \Delta \phi(\rho) \qquad \text{in }[0,\infty) \times \R^d
\end{equation}
can be written as a gradient flow $\dot \rho = K(\rho)\rmD \calS(\rho)$ of the functional 
\begin{equation}
	\label{eqdef:GS-Otto-Wasserstein}
\calS(\rho) :=  -\int_{\R^d} s(\rho(x))\dd x, \qquad\text{with } \ u \,s''(u) =  \phi'(u),
\end{equation}
with the same operator $K(\rho)$ as in~\eqref{eqdef:S-K-JKO}.

Dirr, Stamatakis, and Zimmer~\cite[(7)]{DirrStamatakisZimmer16} derive the equation~\eqref{eq:nonlinear-diffusion-PDE} as the limit of the \emph{zero-range process}, with a similar functional $\calS$ but a different choice of $s$ and~$K$:
\[
s'(u) := \log \phi(u)\qquad\text{and}\qquad 
K(\rho)\xi := -\div \pra[\big]{\phi(\rho) \nabla \xi}\,.
\]
The authors refer to $s$ as an entropy.
In models of heat conduction, again the same equation appears but with yet another gradient structure~\cite{PeletierRedigVafayi14}:
\begin{equation}
	\label{eqdef:s-heat-conduction}
s(u) = -\log u
\qquad\text{and}\qquad
K(\rho)\xi := -\div \pra[\big]{\rho^2 \nabla \xi}\,,
\end{equation}
where $s$ is  again called an entropy by the authors.

These examples illustrate something that is well known in the \Generic and gradient-flow communities: the same partial differential equation can have multiple gradient-flow or \Generic structures, showing that the answer to question~\ref{q:unique-char-of-GENERIC} is negative (see also~\cite[Cor.~3.3.1]{Mielke16a} and the discussion of~\ref{q:unique-char-of-GENERIC} in Section~\ref{s:conclusion}). This immediately gives rise to the question `what determines the structure'? \MARK{Do we want to say more? E.g. we can add $\calE$ to $\calS$ without changing the equation, but then we no longer satisfy the NIC?}
\end{Ex}

\begin{Ex}[The Boltzmann equation as \generic system]
\label{ex:hom-BE}

Various authors have shown that the Boltzmann equation,
\[
\partial_t f = -\div_x (vf) +  \int_{\R^d} \!\int_{\bbS^{d-1}} \pra[\big]{f'f'_*-ff_*} B(v-v_*,\omega)
\dd v_*\dd \omega \qquad \text{in }\R^{2d},
\]
(see e.g.~\cite{Erbar24} for the notation)
can be written in \Generic form $\dot f = J(f) \rmD \calE(f) + K(f) \rmD \calS(f)$, 
where 
\[
\calE(f) := \frac12 \int_{\R^{2d}} |v|^2 f(x,v)\dd xdv 
\qquad \text{and}\qquad 
\calS(f) := -\int_{\R^{2d}} f \log f. 
\]
The Poisson operator is given by 
\[
J(f) \xi := \div_v \bra[\big]{f\nabla_x \xi} - \div_x \bra[\big]{f\nabla_v \xi},
\]
and for $K = K(f) \xi$ different expressions have been given, both linear in $\xi$~\cite[(11)]{Ottinger97} (see also~\cite{Erbar24}) and nonlinear in $\xi$~\cite[(A7)]{Grmela93}.
\end{Ex}

\smallskip
Many more systems can be written in \generic form, such as  the Vlasov-Fokker-Planck equation~\cite{DuongPeletierZimmer13},  reaction-diffusion systems~\cite{Mielke11}, plasma dynamics~\cite{Morrison86}, and many others. Our aim here is not to be exhaustive but to give some understanding, and we therefore limit ourselves to the examples above.


\bigskip

\subsection{The basic answer}
\label{ss:basic-answer}

We return to the basic question Q0,  `Why does entropy drive an	 evolution?'
An essential concept in this discussion is that of \emph{coarse-graining}, and a corresponding \emph{coarse-graining map}. 
\begin{definition}
	Let $\calX$ and $\calY$ be sets. 
	Any {non-injective} map $\Phi:\calX\to\calY$ is called a \emph{coarse-graining map}. \emph{Coarse-graining} is the mapping under $\Phi$ of an evolution in $\calX$ into an evolution in $\calY$.
\end{definition}
\noindent
It is exactly the non-injectivity of this map that gives rise to the concept of entropy. 

\medskip
In this paper we formulate a single answer to both this basic question and all the related questions listed above. 

\begin{answer}
\label{ans:basic}
Entropy drives evolution because it characterizes the invariant measure of an underlying `microscopic' dynamics after {coarse-graining}. 
Entropy itself should be understood as either `equal to' or `a remnant of' this  invariant measure, depending on the situation. 
The entropy and the `driving' are generated together by the coarse-graining.
\end{answer}

\noindent
The purpose of the rest of this paper is  both to explain this answer and to show how it unites the different examples.

\bigskip


\medskip

The Basic Answer above describes two ways in which entropy may appear as a driver of a given evolution equation. Both ways involve connecting the equation to some `microscopic dynamics', which often is stochastic, and this stochasticity expresses itself in different ways. In the first interpretation---entropy being `equal' to an invariant measure---the stochasticity is visible in the evolution equation that one started with, while in the second interpretation---entropy being a `remnant' of an invariant measure---the stochastic process is only visible `in the background'.

The examples in Section~\ref{ss:intro-other-examples} are all deterministic evolution equations, and in this sense they fall into the second category. Here the stochastic process only plays a role in the background of the evolution equation, and for several of them the entropy has an interpretation as a large-deviation rate function. We will discuss such cases in Section~\ref{s:entropy-large-deviations}.

However, I claim that the first interpretation, in which the evolution is random, is actually the more fundamental one, and that the deterministic examples should be considered to be deterministic limits of such random evolutions. We will first explore this basic idea in Section~\ref{s:core-example}, and we will see this in detail for Example~\ref{ex:damped-harmonic-oscillator} in Section~\ref{s:damped-pendulum}.

\begin{remark}[Entropy as volume of macrostates]
Many readers will be familiar with the abstract idea that entropy is the `volume of a region of microstates corresponding to a single macrostate', going back to Boltzmann~\cite{Boltzmann77b}. This `volume' point of view should be seen as a particular instance of the more precise interpretation that we discuss in this paper. 
\end{remark}

\begin{remark}[Entropy and free energy]
	\label{rem:entropy-free-energy}
`Entropy' is closely connected to the various interpretations of `free energy', and there is much that could be said about this connection. In this paper I have decided to mostly leave this aspect aside, and only make a few remarks:
\begin{itemize}
\item Although historically and in other contexts there is good reason to distinguish `entropy' from `free energy', I claim that for the purpose of this paper---understanding why certain functionals appear to drive evolution equations---it does not matter which name we give to such a functional. 
\item In the statistical mechanics community, it is common to call any large-deviation rate function an `entropy' (at least in my personal experience). Given the role of large deviations in the Basic Answer and in Section~\ref{s:entropy-large-deviations} below, this supports the choice not to stress the distinction between entropy and free energy. 
\end{itemize}
For this reason I also call functionals such as in row~\ref{ent:FP} of Table~\ref{table:different-entropies}  `entropies' even though others use the term `free energy' (e.g.~\cite{JordanKinderlehrerOtto97,JordanKinderlehrerOtto98}).
\end{remark}

\begin{remark}[Novelty]
	This paper contains few novel results or insights, in the usual sense of the word; for most of the remarks and statements that I make, there will be  researchers who feel that it is `well known'. On the other hand, I do believe that this particular gathering and organization of insights around the questions of this paper is new, and that it may be of use to many mathematicians working in the fields of gradient flows and GENERIC systems. 
\end{remark}

\section{A core example}
\label{s:core-example}

\MARK{From notes: Can we give a convincing and simple example like this Core Example that does not require $V = V\circ \Phi$?}

We start with an example that illustrates the Basic Answer of Section~\ref{ss:basic-answer} in a very simple setup, in which the microscopic dynamics is stochastic and has the form of a stochastic differential equation (SDE). In this example we can see in detail three ingredients of how the Basic Answer arises:
\begin{enumerate}[label=(\alph*)]
	\item \label{enum:ex-principle1}
	 Coarse-graining an SDE leads to drift terms in the SDE;
	\item The drift terms involve derivatives of the coarse-grained invariant measure, and can also be expressed in terms of the \emph{degeneracy} (see Remark~\ref{rem:Sn-is-degeneracy}) of the coarse-graining map;
	\item \label{enum:ex-principle3}
	The drift terms are generated by the interaction between the degeneracy and the noise.
\end{enumerate}
We discuss below how this degeneracy is effectively called \emph{entropy.}
While this example is special, it illustrates the basic principles behind all of the examples of this paper.

\subsection{Setup of the core example}
\label{ss:setup-core-example}

The setup is that of a `microscopic' process $X$ in $\R^n$ and its coarse-grained counterpart $Y_t := \Phi_n(X_t) = |X_t|$:
\begin{equation}
	\label{eq:X-to-Y-coarse-graining}
\left\{
		X_t \text{ process in }\R^n
\right\}
\quad
\xrightarrow{\quad \Phi_n(x) \,:=\, |x|\quad }
\quad 
\left\{
		Y_t \text{ process in }\R_+
\right\}.
\end{equation}
The process $X$ is defined to be flat diffusion in the unit ball $B(0,1)$ in $\R^n$, with mobility scaled by a parameter $\beta>0$:
\begin{equation}
	\label{eq:CG-example-SDE-X}
\rmd X_t = \sqrt{\frac2\beta } \,\rmd W_t \qquad\text{in }B_{\R^n}(0,1), 
\qquad\text{with reflecting boundary conditions}.
\end{equation}
Here $W$ is an $n$-dimensional Brownian motion. This process $X$ has a unique stationary probability measure 
\begin{equation}
	\label{eqdef:CG-example-pi_n}
	\mu_n := c_n\Lebesgue^n|_{B(0,1)},
\end{equation}
where $\Lebesgue^n$ is the $n$-dimensional Lebesgue measure and $c_n>0$ is a normalization constant.

The coarse-grained process $Y_t= \Phi_n(X_t) = |X_t|$ is again a Markov process and It\^o's Lemma implies that it solves the SDE~\cite[Def.~XI.1.9]{RevuzYor13}
\begin{align}
\rmd Y_t &= \frac{n-1}{\beta Y_t} \dd t + \sqrt{\frac2\beta} \dd B_t\notag  \\
&=\frac1{\beta}\calS_n'(Y_t) \dd t + \sqrt{\frac2\beta} \dd B_t\,.	
\label{eq:CG-example-SDE-Y}
\end{align}
where now $B$ is a one-dimensional Brownian motion, and where in the second line we used the new function
\begin{equation}
	\label{eqdef:CG-example-S}
\calS_n(y) := (n-1)\log y \qquad \text{for }0< y< 1\,.
\end{equation}
The reflecting boundary condition at $|X|=1$ becomes a reflecting boundary condition at $Y=1$.

\medskip
We claim that the appearance of the new drift term in~\eqref{eq:CG-example-SDE-Y}, which is characterized by the function $\calS_n$ and the noise parameter $\beta$, gives a basic insight into what entropy is and why it drives evolution equations.

\medskip
The next lemma gives some first indications; we first state and prove the lemma and afterwards discuss its interpretations.

\begin{lemma}
	\label{l:CG-example-S_n}
	Define the process $X$ as in~\eqref{eq:CG-example-SDE-X} and set $Y_t := \Phi_n(X_t) = |X_t|$ for all $t$, as above.
	The process $Y$ has the invariant measure $\nu_n := (\Phi_n)_\# \mu_n$ (see~\eqref{eqdef:push-forward}), which can also be characterized as 
	\[
	\nu_n(\rmd y) = c_n'\,\ee^{\calS_n(y)} \rmd y,
	\]
	where $c_n'$ is a normalization constant.
The function $\calS_n$ satisfies
\begin{subequations}
	\label{l:CG-example-S_n-char}
\begin{align}
	\calS_n(y) &= \log \frac{\rmd (\Phi_n)_\# \mu_n}{\rmd\Lebesgue^1}(y) + \mathrm{constant}_{1,n}
	\label{l:CG-example-S_n-char1}\\
	&= \log \int_{\R^n} \pdelta\pra[\big]{\Phi_n(x) - y}(\rmd x)  + \mathrm{constant}_{2,n}.
	\label{l:CG-example-S_n-char2}
\end{align}
\end{subequations}
\end{lemma}
\noindent
Here $\pdelta$ is the measure on $\R^n$ that is heuristically defined for any $f:\R^n\to\R$ and $y\in \R$ by
\begin{equation}
	\label{eqdef:pdelta}
\int_{\R^n} \varphi(x) \pdelta\pra[\big]{f(x) - y}(\rmd x) := 
  \lim_{h\downarrow 0} \frac1h \int_{\R^n} \varphi(x) \bONE\big\{y \leq f(x) < y+h\big\}\dd x
  \qquad \text{for all }\varphi\in C_c(\R^n).
\end{equation}
The measure $\pdelta\pra[\big]{f(x) - y}$ is concentrated on the level set $f^{-1}(y)$ in $x$-space.
Such measures are often called `microcanonical' in view of their use to describe `microcanonical ensembles' in statistical physics. 
In Appendix~\ref{app:pdelta} we give a rigorous definition and we specify the properties that we use.

\begin{proof}[Proof of Lemma~\ref{l:CG-example-S_n}]
The fact that $Y$ has the push-forward measure as invariant measure is a simple consequence of the properties of the push-forward. 
The characterization~\eqref{l:CG-example-S_n-char1} follows by applying~\eqref{eqdef:push-forward} to $\Phi_n$ and $\mu_n$, finding for any measurable $\varphi:[0,1]\to\R$ that
\begin{align*}
\int_0^1 \varphi(y) \, \nu_n(\rmd y) \ 
&\leftstackrel{\eqref{eqdef:push-forward}}= c_n \int_{|x|<1} \varphi(|x|) \dd x 
= c_n \int_0^1 \varphi(r)\, \calH^{n-1}(\partial B(0,r)) \, \dd r\\
&= c_n \int_0^1 \varphi(r) C_n r^{n-1}\, \dd r
= c_n'\int_0^1 \varphi(r) \exp\bra[\big]{\;\underbrace{(n-1) \log r}_{\calS_n(r)} \,}\dd r.
\end{align*} 
The expression~\eqref{l:CG-example-S_n-char2} is a simple rewriting of~\eqref{l:CG-example-S_n-char1}, which can be recognized by noting that 
\begin{align*}
\frac{\rmd (\Phi_n)_\# \mu_n}{\rmd\Lebesgue^1}(y)
&= \lim_{h\downarrow0} \frac1h (\Phi_n)_\#\mu_n\bra[\big]{[y,y+h)}
= \lim_{h\downarrow0} \frac1h \mu_n\bra[\big]{\set*{x\in\R^n: y \leq \Phi_n(x)< y+h}}.
\end{align*}
Since $\mu_n$ is a constant times $n$-dimensional Lebesgue measure, this final expression coincides with the right-hand side of~\eqref{eqdef:pdelta} when applied to~\eqref{l:CG-example-S_n-char2}.
\end{proof}


\subsection{Interpretation of the core example}

I used the letter $\calS$ for the function $\calS_n$ in~\eqref{eqdef:CG-example-S} because this function is the first example of the class of `entropies' according to the Basic Answer of Section~\ref{ss:basic-answer}. We discuss two important features. 

\paragraph{$\calS_n$ `is equal to' the invariant measure after coarse-graining.} 
The flat invariant measure $\mu_n = c_n\Lebesgue^n$ for $X$  transforms under $\Phi_n$ to the tilted measure 
\begin{equation}
	\label{eqdef:inv-meas-core-example-e-Sn}
	c_n' \,\ee^{\calS_n(y)}\rmd y,
\end{equation}
and by construction this measure is invariant for $Y_t = \Phi_n(X_t)$. In this sense $\calS_n$ `is equal to' the invariant measure after coarse-graining,  up to taking a Lebesgue density and a logarithm. 

\paragraph{The drift term $\calS_n'(Y_t)/\beta$ in~\eqref{eq:CG-example-SDE-Y} is caused by the combination of  coarse-graining and noise.}
This example also shows that the drift in the SDE~\eqref{eq:CG-example-SDE-Y} is not caused by the coarse-graining  alone, but by the combination of both coarse-graining and noise. At a formal level, this can be observed in the appearance of the noise parameter~$\beta$ in the drift term $\calS_n'/\beta \dd t$ in~\eqref{eq:CG-example-SDE-Y}; smaller noise  (larger $\beta$) implies a smaller drift.

\bigskip\noindent
This example and the function $\calS_n$ therefore illustrate the two components of the Basic Answer:
\begin{itemize}
\item \emph{Entropy should be understood as either `equal to' or `a remnant of' an invariant measure.} In this case the `equal to' applies: $\calS_n$ is (the logarithm of the Lebesgue density of) the coarse-grained invariant measure $\nu_n$.
\item \emph{Both this invariant measure and the `driving' by the entropy are generated by {coarse-graining} more microscopic dynamics.} Indeed, both the invariant measure $\nu_n$ and the drift term $\calS_n'(Y_t)/\beta$ in~\eqref{eq:CG-example-SDE-Y} are shaped by both (a) the microscopic dynamics of $X$ and (b) the coarse-graining $Y = \Phi_n(X)$.
\end{itemize}

\begin{remark}[$\calS_n$ is called the \emph{degeneracy} of $\Phi_n$.]
	\label{rem:Sn-is-degeneracy}
The function $\Phi_n$ is non-injective, and the size of level sets $\Phi_n^{-1}(y)$ varies with $y$. The characterizations~\eqref{l:CG-example-S_n-char} show that $\calS_n$ is equal to this `size of the level set', up to a logarithm and an additive constant.  In the literature the size of a level set of a coarse-graining map such as $\Phi_n$ is often called the `degeneracy' of the set, and with this terminology $\calS_n$ is the logarithm of the degeneracy of $\Phi_n$.
\end{remark}

\paragraph{Geometry gives another view on emergence of drift.} 
Figure~\ref{fig:geometry} illustrates geometrically the same point as the observation above, that noise and coarse-graining together convert the unbiased diffusion process $X$ into a process $Y$ with drift. The figure shows how it is the combination of \emph{noise} with \emph{curvature of the level sets} of $\Phi_n$ that generates this drift. This is consistent with the observation above that variation in size of level sets generate the drift, since the curvature of a level set exactly characterizes the local variation in level-set area between neighbouring level sets. 
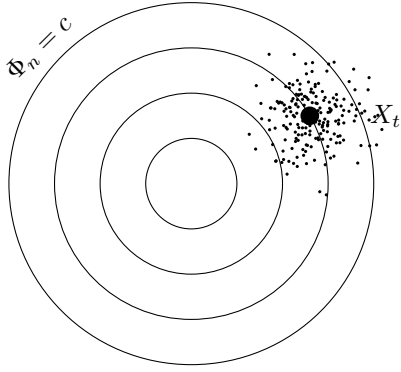
\begin{figure}[ht]

\centering
\begin{tikzpicture}[scale=1]
	\pgfmathsetseed{123456}

	\draw[thin] (0,0) circle (0.6);
	\draw[thin] (0,0) circle (1.2);
	\draw[thin] (0,0) circle (1.8);
	\draw[thin] (0,0) circle (2.4);

	\fill[black] (1.55884572681,0.9) circle (3.5pt);

	\node[rotate=45] at (-2,1.8) {\footnotesize $\Phi_n = c$};

	\begin{scope}[shift={(1.55884572681,0.9)}]

		\def\sigma{0.35}     
		\def\npoints{200}    

		\foreach \i in {1,...,\npoints} {
			\pgfmathsetmacro{\u}{rnd}
			\pgfmathsetmacro{\v}{rnd}
			\pgfmathsetmacro{\zone}{sqrt(-2*ln(\u))*cos(360*\v)}
			\pgfmathsetmacro{\ztwo}{sqrt(-2*ln(\u))*sin(360*\v)}
			\pgfmathsetmacro{\x}{\sigma*\zone}
			\pgfmathsetmacro{\y}{\sigma*\ztwo}
			\fill (\x,\y) circle (0.6pt);
		}

		\node at (1,0) {\footnotesize $X_t$};
	\end{scope}

\end{tikzpicture}

\caption{
A particle $X$ released at time zero at $X_0$ (the large dot) visits neighbouring points as time increases (the small dots). These points are symmetrically spread around $X_0$ in $X$-space, but the curvature of the level sets of $\Phi_n$ causes their $Y$-values to be biased towards larger values of $\Phi_n$. }
\label{fig:geometry}
\end{figure}

\subsection{Non-flat diffusion and connection to `free energy'}

In the example of the previous section the process $X$ has the Lebesgue measure as invariant measure. We now modify this to a measure of the form 
\[
\mu_n(\rmd x) = \ee^{-\beta V(x)}\dd x 
\qquad\text{with }V(x) = \widetilde V (|x|).
\]
This measure is invariant if we add a drift term to~\eqref{eq:CG-example-SDE-X}, 
\begin{equation}
\rmd X_t  = - \nabla V(X_t) \dd t + \sqrt{\frac2\beta }\dd W_t.
\end{equation}
Performing the same coarse-graining by $\Phi_n$ in~\eqref{eq:X-to-Y-coarse-graining} we find that $Y_t = \Phi_n(X_t) = |X_t|$ satisfies the corresponding SDE, which generalizes~\eqref{eq:CG-example-SDE-Y},
\begin{equation}
\rmd Y_t = -  \widetilde V'(Y_t)\dd t + 
\frac1{\beta}\calS_n'(Y_t) \dd t + \sqrt{\frac2\beta} \dd B_t.
\end{equation}
The analogue of Lemma~\ref{l:CG-example-S_n} for this situation is as follows. Note that we now use a different symbol (`$\calF_n$') and name (`free energy') for the logarithmic Radon-Nikodym derivative; this allows us to separate contributions from $V$ (`energy') from those coming from the degeneracy (`entropy'). However, as remarked above, one should not attach too much significance to these names---the aim is to understand how degeneracy enters the situation.
\begin{lemma}
	\label{l:CG-example-S_n-with-V}
	The process $Y$ has the invariant measure $\nu_n := (\Phi_n)_\# \mu_n$, which can also be characterized as 
	\[
	\nu_n(\rmd y) = c_n'\,\ee^{-\beta \calF_n(y)} \rmd y,
	\]
	where $c_n'$ is a normalization constant.
The `free energy' function $\calF_n$ satisfies
\begin{subequations}
\begin{align}
	-\beta \calF_n(y) &= \log \frac{\rmd (\Phi_n)_\# \mu_n}{\rmd\Lebesgue^1}(y) + \mathrm{constant}_{1,n}
	\\
	&= \log \int_{\R^n} \ee^{-\beta V(x)}\pdelta\pra[\big]{\Phi_n(x) - y}(\rmd x)  + \mathrm{constant}_{2,n}.
\end{align}
\end{subequations}
We also have the relation
\[
\calF_n(y) = \wt V(y) - \frac1\beta \calS_n(y),
\]
where $\calS_n$ is given in~\eqref{eqdef:CG-example-S}.
\end{lemma}
\noindent
The proof is very similar to that of Lemma~\ref{l:CG-example-S_n} and we omit it.

\subsection{Further comments}

\begin{remark}[Combined large-dimension and small-noise limits]
	\label{rem:deterministic-equations-from-combined-limits}
	There is an apparent contradiction in the discussion above: If the drift generated by degeneracy depends on noise, then why does one see these effects also in deterministic equations such as in the examples in Section~\ref{se:central_question}? The answer lies in a combined limit $\beta\to\infty$ and $n\to\infty$: if we consider a sequence of systems of increasingly high dimension $n\to\infty$, and corresponding vanishing noise levels $\beta_n \to\infty$,  in such a way that the ratio $\beta_n/n$ converges to a constant~$\beta_\infty$, then we observe that the drift term in the SDE~\eqref{eq:CG-example-SDE-Y} converges to
\[
	\frac{n-1}{\beta_n Y_t} \;\xrightarrow{n\to\infty} \;\frac1{\beta_\infty Y_t} \;= \;\frac1{\beta_\infty} \calS'(Y_t),
\]
where we define the limiting entropy-per-degree-of-freedom $\calS$ by
\begin{equation}
	\label{eq:conv-Sn-S}
\calS(y) := \lim_{n\to\infty } \frac 1n \calS_n(y) = \log y.
\end{equation}
As $n\to\infty$ the process~$Y$ therefore converges in distribution to the solution of the deterministic equation
\begin{equation}
\label{eq:limit-ODE}
	\rmd Y_t = \frac{1}{\beta_\infty Y_t} \dd t 
	\qquad\text{or equivalently}\qquad 
	\dot y = \frac1{\beta_\infty y},
\end{equation}
which is a gradient-flow equation of the form~\eqref{eq:GF},
\[
\dot y = K \calS'(y), \qquad \text{with} \quad K: \R\to\R, \quad K\xi := \frac1{\beta_\infty}\xi.
\]
Note how the operator $K$ explicitly depends on the characterization $\beta_\infty$ of the scaling of the noise in $n$. (Also note that the reflecting boundary conditions mean that~\eqref{eq:limit-ODE} only holds up to the time that $y$ reaches $y=1$.)

This remark shows how the combined small-noise, large-dimension limit leads to a deterministic gradient flow with drift $K\calS'$ determined by both the noise and the degeneracy. This also is an example of the second option in the Basic Answer of Section~\ref{ss:basic-answer}. 
\end{remark}

\begin{remark}[The scaling $\beta_n/n = O(1)$ is common]
While the choice $\beta_n/n \to \beta_\infty\in (0,\infty)$ in Remark~\ref{rem:deterministic-equations-from-combined-limits} may appear arbitrary, it is a simplified version of a very common situation. In many cases $\Phi_n$ is an average of $n$ objects, as in e.g.~\eqref{eqdef:Phi_n-empirical-measure} below, and the noise in each of these objects is therefore scaled by $1/n$; the choice $\beta_n/n\to \beta_\infty$ above is an artificial version of this. In fact, in all of the examples in Section~\ref{s:entropy-large-deviations} the function~$\Phi_n$ is some form of average, and those examples show the same basic behaviour as in Remark~\ref{rem:deterministic-equations-from-combined-limits}.
\end{remark}

\MARK{
\begin{remark}[Drift terms from bijective transformations]
	Need to comment on drift terms that arise from bijective transformations; they are traditionally not called `entropic'.
\end{remark}
}

\section{Entropy as `remnant' of invariant measures}
\label{s:entropy-large-deviations}

In the previous section we showed that coarse-graining an SDE leads to  drift terms involving functions $\calS_n$  that we called `entropy', and which characterize an invariant measure.

Although the coarse-grained evolution is still stochastic, it may become deterministic in a combined small-noise and large-dimension limit, as described in Remark~\ref{rem:deterministic-equations-from-combined-limits}. 
This is the case of the second option of the Basic Answer of Section~\ref{ss:basic-answer}, which states that entropy may be a `remnant' of an invariant measure. The more general and more precise statement is that entropy may be a \emph{large-deviation rate function} of a sequence of invariant measures. We describe entropy in this case as a `remnant', since it is not the density of the invariant measure itself, but something that `remains' of the invariant measure after taking a singular limit.

In a very heuristic sense, the two cases can be compared as follows:
\begin{align*}
&\text{Option 1, stochastic case: } && \text{invariant measure }\ = \ee^{\calS(y)}\dd y\\
&\text{Option 2, large-deviation case: } && \text{invariant measure }\ \sim \ee^{n\calS(y)}\dd y\quad \text{in the limit }n\to\infty.
\end{align*}
In the limit in Remark~\ref{rem:deterministic-equations-from-combined-limits} one can see how the system transitions from the first case to the second. At finite $n$, the invariant measure $\ee^{\calS_n(y)}\dd y$ in~\eqref{eqdef:inv-meas-core-example-e-Sn} is of the first type, but the scaling $\calS_n(y) \sim n\calS(y)$ given by~\eqref{eqdef:CG-example-S} means that in the limit $n\to\infty$ the situation becomes that of the second type.

\medskip

In this section we explore this second interpretation, in which the entropy is a  large-deviation rate function of invariant measures.

\subsection{Relative entropies and Sanov's theorem}
\label{ss:relative-entropies-Sanov}
Behind the `$\rho\log \rho$'-type terms in Table~\ref{table:different-entropies} lies the theorem of Sanov, which gives a general characterization of the large-deviation properties of empirical measures. To describe this we first introduce a class of \emph{relative entropies}.

Let $\phi: \R\to\R\cup\{+\infty\}$ be convex, and let $\Omega$ be a topological space. 
\begin{definition}
\label{def:RelEnt}
The $\phi$--\emph{relative entropy} of two probability measures $\mu$ and $\nu$ on $\Omega$ is defined as
\begin{equation}
	\label{eqdef:RelEnt}
\RelEnt_\phi(\mu|\nu) := \begin{cases}
\ds  \int_\Omega \phi\bra*{\frac{\dd \mu}{\dd\nu}(x)}\nu(\rmd x)
  & \text{if $\mu\ll \nu$},\\
+\infty & \text{otherwise}.
\end{cases}
\end{equation}
The \emph{Boltzmann relative entropy} (also called \emph{Kullback-Leibler divergence}, or simply \emph{relative entropy}) is the special case
\[
\phi(s) = \phi_{\mathrm B}(s) := \begin{cases}
	s\log s &\text{if }s>0,\\
	0 & \text{if }s=0,\\
	+\infty & \text{if }s<0.
\end{cases}\qedhere
\]
\end{definition}
\MARK{Do we actually need the full generality of $\calH_\phi$? Do we actually use it later? Otherwise we could also just make a remark about making the formulae rigorous.}

To illustrate the appearance of the `$\rho\log\rho$'-terms in Table~\ref{table:different-entropies} we describe a simple probabilistic experiment. Fix an open set $\Omega\subset \R^d$ and a function $V:\Omega\to\R$. For simplicity we assume that $\mu(\rmd x) := \ee^{-V(x)}\dd x$ is a probability measure, i.e.\ that $\int_\Omega \ee^{-V(x)}\dd  x=1$. 

Let $X_1,X_2,X_3,\dots$ be a sequence of i.i.d.\ random variables in $\Omega$ with common law $\mu$, and define their empirical measures
\[
\rho^n := \frac1n \sum_{i=1}^n \delta_{X_i},
\]
or more formally in terms of a coarse-graining map $\Phi_n$,
\begin{subequations}
	\label{eqdef:Phi_n-empirical-measure}
\begin{align}
\rho^n &:= \Phi_n(X_1,\dots, X_n), \qquad\text{with}\quad \\
\Phi_n&:	\Omega^n \to \ProbMeas(\Omega), \quad \Phi_n(x_1,\dots ,x_n):=\frac1n\sum_{i=1}^n \delta_{x_i}.
\end{align}
\end{subequations}
For fixed $n$, $\rho^n$ is a random probability measure on $\Omega$, i.e.\ a random element of $\ProbMeas(\Omega)$.

In the limit $n\to\infty$, the following things happen:
\begin{enumerate}
\item Law of large numbers: $\rho^n$ converges almost surely to the law $\ee^{-V(x)}\dd x$ of the $X_i$.
\item Large deviations (see Theorem~\ref{t:Sanov} below): As $n\to\infty$, for each $A\subset \ProbMeas(\Omega)$, 
\begin{equation}
	\label{eq:LDP-intro}
\Prob\pra[\big]{\rho^n\in A} \sim \ee^{-nI(A)}, \qquad I(A) := \inf_{\rho\in A} I(\rho), \qquad I(\rho) := \int_\Omega\rho(\log \rho + V). 
\end{equation}
\end{enumerate}
Informally, the statement~\eqref{eq:LDP-intro} means that for each $\rho\in \ProbMeas(\Omega)$ the probability that $\rho^n$ is close to $\rho$ scales as 
\[
\Prob\pra[\big]{\rho^n\approx \rho} \sim \ee^{-nI(\rho)} \qquad \text{as }n\to\infty.
\]
Note that the \emph{rate function} $I$ is the Boltzmann relative entropy with respect to the law~$\mu$ of the $X_i$: writing $\rho(\rmd x) = \rho(x) dx$, 
\begin{align*}
\RelEnt_{\phi_{\mathrm B}}(\rho|\mu) 
&= \int_\Omega \phi_{\mathrm B}\bra*{\frac{\dd\rho}{\dd\mu}} \dd\mu
= \int_\Omega \phi_{\mathrm B}\bra*{\rho(x) \ee^{V(x)}} \ee^{-V(x)}\dd x\\
&= \int_\Omega \bra[\big]{\rho(x)\log \rho(x) + \rho(x) V(x)}\dd x = I(\rho).
\end{align*}
Also note that $I(\mu) = 0$, and this is consistent with the law of large numbers: if $A\subset \ProbMeas(\Omega)$ is an open set containing $\mu$, then the law of large numbers implies $\Prob[\rho^n\in A]\to1$, and this is matched by the fact that $I(A)=I(\mu) = 0$.

The large-deviation statement above is a special case of Sanov's theorem, which applies in much more generality:

\begin{theorem}[Sanov's theorem {\cite[Sec.~6.2]{DemboZeitouni98}}]
\label{t:Sanov}
Let $X_1,X_2, \dots$ be i.i.d.\ random variables taking values in a complete separable metric space $\bfX$ with common law $\mu\in \ProbMeas(\bfX)$. Define the empirical measure of the first $n$ random variables, 
\[
\rho^n := \frac1n \sum_{i=1}^n \delta_{X_i}.
\]
Then the random variable $\rho^n$ satisfies a large-deviation principle with rate function 
\begin{equation}
	\label{eqdef:I-LDP-Sanov}
I(\rho) := \RelEnt_{\phi_{\mathrm B}}(\rho|\mu).
\end{equation}
\end{theorem}
\noindent
Much more can be found on large-deviation theory in e.g.~\cite{DeuschelStroock89,DemboZeitouni98,DenHollander00}.

\bigskip

With Sanov's theorem we can explain the form of the expressions in rows~\ref{ent:diffusion} and \ref{ent:FP} of Table~\ref{table:different-entropies}, as follows. In Example~\ref{ex:JKO} the Fokker-Planck equation~\eqref{eq:FPeq} describes the law of a Brownian particle in an energy landscape, which follows the SDE in $\R^d$
\begin{equation}
	\label{eqdef:Stoch-proc-X-FPeq}
\rmd X_t = -\nabla V(X_t)\dd t + \sqrt 2 \dd W_t.
\end{equation}
This process has an invariant measure on $\R^d$
\[
\mu(\rmd x) = \ee^{-V(x)}\dd x.
\]
As above, for simplicity we assume  that $\mu$ is normalizable and has unit mass.

Taking a sequence of independent copies $X_1, X_2, \dots$ of this process, each $X_i$ has the same invariant measure $\mu$; if each of them indeed is distributed according to $\mu$, then by Theorem~\ref{t:Sanov} and the discussion above we find that $\rho^n$ satisfies a large-deviation principle with rate function
\begin{equation}
	\label{eqdef:I-entropy-Sanov-example-FPeq}
I(\rho) = \RelEnt_{\phi_\mathrm B}(\rho|\mu) 
= \int_{\R^d} \rho(x)\bra[\big]{\log \rho(x) + V(x)}\dd x .
\end{equation}
This is the expression in~\eqref{eqdef:S-K-JKO} of Example~\ref{ex:JKO} and in rows~\ref{ent:diffusion} and~\ref{ent:FP} of Table~\ref{table:different-entropies}. 

\begin{remark}[Signs of rate functions and entropies]
Rate functions are by construction non-negative, and the minimal value zero corresponds to the highest-probability behaviour. Under the sign convention of this paper, entropies have the opposite sign, and the highest-probability behaviour corresponds to the \emph{largest} value of $\calS$. Therefore the entropies in Table~\ref{table:different-entropies} resemble the rate functions in this section but all with a minus sign. 
\end{remark}

\subsection{Interacting particles}
\label{ss:interacting-particles}

Sanov's theorem applies to independent particles and the empirical-measure map
\begin{equation}
	\label{eqdef:CG-map-empirical-measure}
(X_1,\dots, X_n) \longmapsto \rho^n := \frac1n \sum_{i=1}^n \delta_{X_i}.
\end{equation} 
If we relax the assumption of independence of the particles, then in many cases a large-deviation principle still holds, with a different rate function. The rate function typically is modified either by adding terms or by changing the function $\phi$. Interactions of `mean-field' type may lead to the addition of convolution integrals, while stronger interaction leads to a different choice of the function $\phi$ in $\RelEnt_\phi$.
We now illustrate this.

\paragraph{Mean-field interaction leads to addition of terms.}
The canonical example of mean-field-interacting particles is an extension of~\eqref{eqdef:Stoch-proc-X-FPeq} to $n$ particles $X^{n,i}$ that evolve with an interaction term that scales as $1/n$, 
\begin{equation}
	\label{eqdef:Stoch-proc-X-VFP}
\rmd X^{n,i}_t = - \frac1n\sum_{j=1}^n \nabla \psi(X_t^{n,i}-X_t^{n,j})\dd t  -\nabla V(X^{n,i}_t)\dd t + \sqrt 2 \dd W^i_t.
\end{equation}
If $\psi$ is smooth and bounded, then under the same integrability condition on $\ee^{-V}$ this process again has a stationary measure on the space $(\R^d)^n$ of $n$-particle positions, and consequently the empirical measure $\rho^n$ of these particles also has a stationary measure. As $n\to\infty$, in stationarity, $\rho^n$ also satisfies a large-deviation principle, with rate function
\begin{equation}
I(\rho) = \int_{\R^d} \rho(x)\bra[\Big]{\log \rho(x) + V(x) + \frac12 (\psi{*}\rho)(x)}\dd x.
\end{equation}
(For smooth and bounded $\psi$ this follows from Sanov's theorem and Varadhan's lemma~\cite[Th.~III.13]{DenHollander00}; see e.g.~\cite{ChafaiGozlanZitt14,Garcia-Zelada19,DupuisLaschosRamanan20} for generalizations.)
Note how including the interaction term in~\eqref{eqdef:Stoch-proc-X-VFP} leads to the \emph{addition} of a term in $I$; in particular, the $\rho\log\rho$-nature of the first term is not changed. This is a hallmark of mean-field or \emph{weak} interaction, in which each particle interacts with each other particle, with a strength that scales as $1/n$.

\paragraph{Strong interaction leads to different nonlinearities.} When the mean-field interaction described above is replaced by stronger forms of interaction, then the `$\rho\log\rho$'-type terms are replaced by different expressions.

For instance, the \emph{hard-rod system} consists of particles in $\R$ that undergo Brownian motion, except for the condition that they do not approach each other closer than distance $\e>0$. (The name `hard rods' comes from the interpretation that a `particle' is in fact a `hard rod' in $\R$ of length $\e$.) In the limit $n\to\infty$, $\e\downarrow 0$, with $n\e\to a$, Rost~\cite{Rost84} showed that the empirical measure $\rho^n$ of the particles converges to a time-dependent measure $\rho(t)$ on $\R$ that satisfies an equation of the form $\partial_t\rho = \Delta \phi(\rho)$. The right-hand side of this equation can be written as $K(\rho)\rmD S(\rho)$ as in~\eqref{eqdef:GS-Otto-Wasserstein}, where 
\[
\calS(\rho) = -\int \rho(x) \log \frac{\rho(x)}{1-a\rho(x)}\dd x
\qquad\text{provided}\quad 0\leq \rho \leq \frac1a \ \text{a.e. on $\R$.}
\]
In~\cite{PeletierGavishNyquist21} we showed that this functional $\calS$ is the large-deviation rate function for the corresponding invariant measure for the particle system, after coarse-graining using the empirical-measure map~\eqref{eqdef:CG-map-empirical-measure}. 

Varadhan~\cite{Varadhan91} also studied systems of diffusion particles with continuous but `strong' interactions: the potential $\psi$ is scaled such that each particle $X_i$ typically interacts with $O(1)$ other particles $X_j$, each with an interaction strength of order $O(1)$. In the limit $n\to\infty$ he also obtained an entropy of the form $- \int s(\rho(x))\dd x$, where this particular function $s$ is determined by a `probabilistic cell problem'. The \emph{Zero Range Process}~\cite{KipnisLandim99,GrosskinskySchutzSpohn03,DirrStamatakisZimmer16} is yet another example where the large-deviation rate function of the empirical measure is of the form $-\int s(\rho(x))\dd x$ for some function $s:\R\to(-\infty,0]$.

\paragraph{Different coarse-graining maps and different invariant measures may lead to different nonlinearities.} The form of large-deviation rate functions depends both on the invariant measure of the particles and on the choice of coarse-graining map. In all the examples above, the stochastic particles move around in space, and their positions are transformed into a probability measure by the coarse-graining map~\eqref{eqdef:CG-map-empirical-measure}. Other processes naturally suggest a different coarse-graining map, and this also gives rise to a different structure for the rate function. 

For instance, the \emph{Brownian Energy Process} with parameter $2$~\cite{RedigVafayi11,PeletierRedigVafayi14} is an SDE for $n$ real-valued stochastic processes $X_i$, $i = 1,2,\dots ,n$. It is a microscopic model for heat conduction, in which energy is randomly exchanged between oscillators located at the points $i$ of a lattice. Its invariant measure is a product of exponential distributions  $\operatorname{Exp}(\eta_0)$ for some $\eta_0>0$, 
\begin{equation}
	\label{eqdef:mn-BEP}
\mu_n (\rmd x_1\cdots\, \rmd x_n)= \prod_{i=1}^n \mu_n^{\mathrm {single}}(\rmd x_i), \qquad\text{where $\mu_n^{\mathrm {single}} := \operatorname{Exp}(\eta_0)$.}
\end{equation}
For this system the coarse-graining map is 
\begin{equation}
	\label{eqdef:CG-map-heat-conduction}
(X_1,\dots, X_n) \longmapsto \eta^n := \frac1n\sum_{i=1}^n X_i \delta_{i/n} ,
\end{equation}
which should be interpreted as placing the energy $X_i$ at position $i/n\in[0,1]$.
Note the difference with~\eqref{eqdef:CG-map-empirical-measure}: here the individual Diracs are fixed in space but change in amplitude, while in~\eqref{eqdef:CG-map-empirical-measure} they are fixed in amplitude but move around in space. 

The invariant measure $\mu_n$ in~\eqref{eqdef:mn-BEP} satisfies a large-deviation principle with rate function~\cite{PeletierRedigVafayi14}
\[
I(\eta) = \int_0^1 \bra*{\frac{\eta(s)}{\eta_0} - 1 - \log \frac{\eta(s)}{\eta_0}}\dd s \;=\; \RelEnt_{\phi}(\eta|\eta_0) \qquad \text{with }\phi(t) = t - 1- \log t.
\]
This function $I$ is of similar form as~\eqref{eqdef:GS-Otto-Wasserstein} and row~\ref{ent:heat-conduction} of Table~\ref{table:different-entropies}.

\subsection{The formula `$\rho \log \rho$' comes from Stirling's approximation}
\label{ss:Stirling}
Although rate functions come in various different forms, as we have seen above, the expression `$\rho\log\rho$' appears so often that it is useful to understand where it comes from. In this section we explore this in a simple example (which is well known; see e.g.~\cite[Sec.~1.6]{Bishop06}). 

Let $J$ be a finite set, and consider $n$ particles in  $J$, i.e. consider a mapping $x:\{1, \dots, n\}\to J$. As in~\eqref{eqdef:Phi_n-empirical-measure} the empirical measure $\rho^n$ is defined as 
\[
\rho^n = \Phi_n(x) := \frac1n\sum_{i=1}^n \delta_{x(i)} \in \ProbMeas(J),
\]
and since $J$ is finite we can write this measure as $\rho^n \simeq(\rho_j)_{j\in J}$, with
\begin{equation}
\label{def:ki}
\rho_j = \frac {k_j}n, \qquad k_j := \#\bigl\{i\in \{1,\dots,n\}: x(i) = j\bigr\}.
\end{equation}
Here $k_j$ is the number of particles at $j\in J$, and $\rho_j$ is the average number of particles at~$j$.

In going from $x$ to $\rho$ there is loss of information;  multiple mappings $x$ produce the same empirical measure $\rho$, or equivalently, the coarse-graining map $\Phi: x\mapsto \rho$ is degenerate. The degree of degeneracy, the number of unique mappings $x$ that correspond to a given~$\rho$, is $n!\,\bigl(\prod_{j\in J}k_j!\bigr)^{-1}$. We now use Stirling's formula in the form
\[
\log m! = m\log m - m + o(m) \qquad\text{as } m\to\infty,
\]
to estimate
\begin{align*}
\log \pra[\bigg]{{n!} \, \bigl(\prod_{j\in J}k_j!\bigr)^{-1}} &= \log n! - \sum_{j\in J} \log k_j!\\
&= n\log n - n - \sum_{j\in J} (k_j\log k_j - k_j) + o(n)\\
&= -n\sum_{j\in J} \rho_j \log \rho_j + o(n) \qquad\text{as }n\to\infty.
\end{align*}
Here we see a similar interpretation as above: the entropy term $-\sum_{j\in J} \rho_j \log \rho_j$ is the rescaled limit, as $n\to\infty$, of the logarithm of the degeneracy.

\medskip
There is a natural probabilistic interpretation of the same insight. 
If we allocate particles independently at random to elements of $J$ with a common law $\mu\in \ProbMeas(J)$, then the probability of obtaining each microstate $x$ with corresponding macrostate $\rho = \Phi_n(x)$ is 
\[
\prod_{i=1}^n \mu_{x(i)} = \prod_{j\in J} \mu_j^{k_j} = \prod_{j\in J} \mu_j^{n\rho_j}.
\]
Since each of the possible microstates $x$ leading to the same macrostate $\rho$, i.e.\ each $x\in \Phi_n^{-1}(\rho)$, has the same probability, the probability of a macrostate $\rho$ satisfies
\begin{align*}
\log \Prob(\rho) &= \log \pra[\bigg]{n!\,\Bigl(\prod_{j\in J}k_j!\Bigr)^{-1}\prod_{j\in J} \mu_j^{n\rho_j}}\\
&=  - n \sum_{j\in J} \rho_j \log \rho_j + n\sum_{j\in J}\rho_j \log \mu_j + o(n) \qquad\text{as }n\to \infty.
\end{align*}
The first two terms on the right-hand side can be combined, leading to a Boltzmann relative entropy with respect to $\mu$:
\[
\sum_{j\in J} \rho_j \log \rho_j - \sum_{j\in J}\rho_j \log \mu_j
= \sum_{j\in J} \rho_j \log \bra[\Big]{\frac{\rho_j}{\mu_j}} = \RelEnt_{\phi_{\mathrm B}}(\rho|\mu).
\]

Summarizing, the expression `$\rho\log\rho$' arises from a counting argument, using Stirling's formula, applied to an empirical-measure-type coarse-graining map. The counting argument is relevant for the probabilistic interpretation because the particles are assumed to be independent.

\begin{remark}[The formula $\rho\log\rho$ also results from additivity of the entropy]
There is an independent argument, going back at least to Shannon~\cite{Shannon48} (see also~\cite[Sec.~33]{Khinchin48}, or for a very clear exposition~\cite[Sec.~1D]{Mackey93}). In this argument `entropy' is postulated to be additive, in the sense that if two systems are formally aggregated into a composite system, without generating any interaction, then the entropy of the composite system should be the sum of the entropies of the separate systems. If `systems' are then reduced to `probability measures on their state spaces' then the additivity property implies that the entropy is essentially the Boltzmann entropy $\calH_{\phi_B}$. Since this argument comes from a completely different starting point than that of this paper, we will not explore this further. 
\end{remark}

\subsection{Gradient flows driven by entropies}
\label{ss:LDP-evolutions}

In the sections above we focused on  the interpretation of entropies as large-deviation rate functions of coarse-grained invariant measures. This showed that differences in the expressions for such entropies can be traced back to differences in the invariant measures of the underlying stochastic processes, or differences in choice of coarse-graining map, or both. 

This does not yet explain \emph{how} such entropies drive the limiting process, i.e.\ where the Onsager operator $K$ comes from that appears in all of the examples of Section~\ref{se:central_question}. The core example of Section~\ref{s:core-example} suggests that it is again the noise that should determine this; the operator $K$ should arise from some limiting process involving the noise. We now discuss another large-deviation result that shows how this is indeed the case, at least for gradient flows and some other examples.

For the case of reversible stochastic processes, leading to gradient flows, we have the following general  result\footnote{This result is formal because an expression of the form~\eqref{eq:I-GGS} requires a rigorous definition of $\calR(z,\dot z)$ and $\calR^*\bigl(z,\tfrac12 \rmD \calS(z)\bigr)$, which depends on the details of the situation. Examples of this are~\cite[Def.~4.1]{DawsonGartner87}, \cite[Ch.~1]{AmbrosioGigliSavare08}, \cite[Th.~8.1]{PeletierGavishNyquist21}, or~\cite[Sec.~2]{PeletierSchlottke22}.} from~\cite[Prop.~3.7]{MielkePeletierRenger14}.
\begin{Ftheorem}[Time-course large deviations lead to gradient structures]
\label{t:MPR}
Let  $Z^n$ be a sequence of continuous-time Markov processes in
a topological vector space $\bfZ$ that are reversible with respect to their stationary measures
$\mu^n\in \ProbMeas(\bfZ)$. Assume that the following two
large-deviation principles hold: 
\begin{enumerate}
\item The invariant measures $\mu^n$ satisfy a large-deviation
  principle with (negative) rate function $\calS:\bfZ\to\R$, i.e.
\[
\mu^n \sim \exp\bra{n\calS}, \qquad \text{as }n\to\infty 
\qquad\text{(note the sign of $\calS$);}
\]
\item The time courses $(Z^n_t)_{0\leq t\leq T}$ satisfy a large-deviation principle in
  $C([0,T];\bfZ)$ with rate function $\scrJ:C([0,T];\bfZ)\to[0,\infty]$,
  i.e.
\begin{equation}
\label{ldp:time-courses}
\Prob\bigl(Z^n \approx z \,\big|\, 
Z^n_0 \approx z(0)\bigr) \sim \exp\bigl(-n\scrJ(z)\bigr),
\qquad \text{as }n\to\infty.
\end{equation}
\end{enumerate}
Then $\scrJ$ can be written as 
\begin{equation}
\label{eq:I-GGS}
\scrJ(z) = \tfrac12 \calS(z(0)) - \tfrac12 \calS(z(T)) + \int_0^T \bigl[ \calR(z,\dot z) + \calR^*\bigl(z,\tfrac12 \rmD \calS(z)\bigr)\bigr]\dd t,
\end{equation}
for some symmetric \emph{dissipation potential} $\calR$.
\end{Ftheorem}

\medskip
A \emph{dissipation potential} $\bfZ\ti\bfZ\ni (z,\dot z)\mapsto \calR(z,\dot z)$ is a function such that for each $z$, $\calR(z,\cdot)$ is convex and lower semicontinuous with $\min \calR(z,\cdot) = \calR(z,0) = 0$. Its Legendre dual is 
\[
\calR^*(z,\xi) := \sup_{\dot z} \langle \xi,\dot z\rangle - \calR(z,\dot z).
\]
By construction, the properties of $\calR$ and $\calR^*$ imply that the functional $\scrJ$ is non-negative, and that curves $z$
satisfying $\scrJ(z)=0$ are solutions of the \emph{generalized} gradient-flow equation
\begin{equation}
  \label{eq:GF-general-MPR}
  \dot z = \rmD_\xi \calR^*(z,\tfrac12 \rmD \calS(z)).
\end{equation}

The large-deviation principle~\eqref{ldp:time-courses} implies that  a
sequence of realizations $Z^n$ of the stochastic process  converges in probability to the zero set of $\scrJ$. The
property $\scrJ(z)=0$ therefore identifies the limiting behavior of the
stochastic process $Z^n$ as solutions of the generalized gradient-flow equation~\eqref{eq:GF-general-MPR}.

In many cases the operator $\xi \mapsto \calR^*(z,\xi)$ is quadratic and non-negative, and can be written as $\calR^*(z,\xi) = \langle \xi,K(z)\xi\rangle$ in terms of some symmetric and non-negative operator~$K$; in that case the generalized gradient-flow equation~\eqref{eq:GF-general-MPR} reduces to the more classic version,
\begin{equation}
	\label{eq:GF-MPR-OnsagerK}
\dot z = K(z) \rmD\calS(z).
\end{equation}
This shows how the Onsager operator $K$ may arise as the dissipation potential $\calR$ in the large-deviation rate function for time courses. In particular, it is the type and covariance structure of the noise that determines $K$, or in the terminology of Maes~\cite{Maes20}, the `frenesy' of the stochastic process.


We now illustrate these ideas  with an example.

\begin{Ex}[Anisotropic Wasserstein gradient flows]
	\label{Ex:anisotropic-Wasserstein}
Consider the following an\-iso\-tropic modification of~\eqref{eqdef:Stoch-proc-X-FPeq}, the process in $\R^d$ given by
\begin{equation}
	\label{eqdef:Stoch-proc-X-FPeq-anisotropic}
	\rmd X_t = -\Sigma\Sigma^\top \nabla V(X_t)\dd t + \sqrt  2 \Sigma\dd W_t.
\end{equation}
Here $\Sigma\in \R^{d\times d}$ is an arbitrary matrix, which modulates the noise $\rmd W_t$ by amplifying it in some directions and reducing it in others. 

The appearance of $\Sigma\Sigma^\top$ in the drift term implies that this equation has the same invariant measure $\mu(\rmd x) = \ee^{-V(x)}\dd x$ as~\eqref{eqdef:Stoch-proc-X-FPeq}, regardless of $\Sigma$. Consequently, following the same steps---defining the entropy as the large-deviation rate function of empirical measures, where the points are i.i.d.\ distributed according to~$\mu$---we find the same characterization
\[
\calS(\rho) = - \int\rho(x)\bra[\big]{\log \rho(x) + V(x)}\dd x.
\]

Applying Theorem~\ref{t:MPR} (see~\cite[Sec.~6(c)]{AdamsDirrPeletierZimmer13}, which builds on the seminal results of Dawson and G\"artner~\cite{DawsonGartner87}) we find that the limiting evolution is a gradient flow of the form~\eqref{eq:GF-MPR-OnsagerK}, where $K$ is the `anisotropic Wasserstein' operator
\begin{equation}
	\label{eqdef:K-Sigma}
	K_\Sigma(\rho)\xi  := -\div \pra[\big]{\rho \Sigma\Sigma^\top \nabla \xi} \qquad \text{ for }\xi:\R^d\to\R.
\end{equation}
Note that for  $\Sigma = I$ this operator coincides with $K$ in~\eqref{eqdef:S-K-JKO}. The resulting evolution equation is 
\[
\partial_t \rho =  \div \pra*{\rho \Sigma\Sigma^\top\nabla  \bra*{\log \rho +  V}} \qquad\text{in } \R^d.\qedhere
\]
\end{Ex}

\subsection{\Generic systems driven by entropies and energies}

Theorem~\ref{t:MPR}  characterizes  gradient-flow structure in terms of large-deviation rate functions. For some \Generic systems this characterization also is known. 

\begin{Ex}[Rate-functional characterization of a \Generic system]
Duong, Peletier, and Zimmer~\cite{DuongPeletierZimmer13} consider the generalized Kramers equation for $\rho\in \calP(\R^{2d})$ and $e \in \R$,
\begin{subequations}
\label{pb:VFPExtended}
\begin{align}
\label{pb:VFPExtended:rho}
\partial_t \rho &= - \div_q\Bigl(\rho \frac pm\Bigr) + \div_p  \rho \Bigl(\nabla_q V + \nabla_q \psi * \rho + \gamma\frac pm \Bigr)+\gamma \theta \Delta_p \rho,\\[\jot]
\frac \rmd{\rmd t} e &= \gamma \int_{\R^{2d}} \frac{p^2}{m^2}\,\rho(dqdp) - \frac{\gamma \theta d}m.
\label{pb:VFPExtended:e}
\end{align}
\end{subequations}
Here $\gamma$, $\theta$, and $m$ are positive parameters, $V$ and $\psi$ are appropriate functions on $\R^{2d}$, and the state space $\R^{2d}$ has coordinates $(q,p)$. The authors show that this system can be written in the \Generic form~\eqref{eq:GENERIC} for the unknown $z:= (\rho,e)\in \calP(\R^{2d})\ti\R$, where the components $J$, $K$, $\calE$, and $\calS$ can all be expressed in the elements of~\eqref{pb:VFPExtended} (see~\cite[(30)]{DuongPeletierZimmer13}).

The authors also show that solutions of~\eqref{pb:VFPExtended} satisfy the equation $J(z)=0$ for the rate function
\begin{equation}
\label{def:J2}
2\theta J(z) := \calS(z(0)) - \calS(z(T)) + \frac12 \int_0^T \Bigl[ \|\dot z-  J \rmD \calE(z) \bigr\|_{ K^{-1}(z)}^2+ \|  \rmD \calS(z)\bigr\|_{ K(z)}^2 \Bigr]  \dd t.
\end{equation}
This rate function describes the large deviations of a more microscopic system of $N$ particles $(Q_i,P_i)$ in the limit $N\to\infty$. In the expression~\eqref{def:J2} one recognizes the same structure~\eqref{eq:I-GGS} of the gradient-flow case, with a particular modification: the potential $\calR(z,\dot z) = \frac12 \|\dot z\|_{K^{-1}(z)}^2$ is applied not simply to the time derivative $\dot z$, but to the `biased' time derivative $\dot z -  J \rmD \calE$ instead. 
\end{Ex}

\subsection{Discussion}

Collecting the insights of this section and the previous one, we recognize a general division of roles:
\begin{enumerate}
\item 
The entropy arises from differences in degeneracy, either at finite noise and parameter levels (Section~\ref{s:core-example} and also in Section~\ref{s:damped-pendulum} below) or in a large-deviation limit (as in this section). In the examples of this section the degeneracy results from the transformation map: for instance, from a vector of positions $(X_1,\dots,X_n)$ to the empirical measure of those $n$ positions, or from a vector of energies $(X_1,\dots,X_n)$ to a discrete `energy measure'~\eqref{eqdef:CG-map-heat-conduction}. 
The typical `$\rho\log \rho$' expression results from Stirling's approximation in the independent case (Section~\ref{ss:Stirling}), and other forms result from strong interaction (Section~\ref{ss:interacting-particles}).

\item The Onsager operator $K$, which converts the derivative of the entropy into a drift in the equation, arises from the structure and the strength of the noise. In Example~\ref{Ex:anisotropic-Wasserstein} this is visible in the noise matrix $\Sigma$ which enters the definition of $K$; in a general sense, the dissipation potentials $\calR$ and $\calR^*$ are a direct reflection of the structure of the noise. Even in the core example of Section~\ref{s:core-example} we observe the dependence of $K$ on the noise,  in the implicit expression $K = 1/\beta$ in the drift term in equation~\eqref{eq:CG-example-SDE-Y}. 
\end{enumerate}

\MARK{We now need to go back and make sure that up to here those questions that can be answered have been answered. Would we want to make an explicit list of all the questions, now interspresed with our answers? }

\section{The example of the damped oscillator}
\label{s:damped-pendulum}

We now turn to Example~\ref{ex:damped-harmonic-oscillator}, which has the \Generic structure
\begin{equation}
	\label{eq:damped-oscillator-eq-repeated}
\frac{\dd}{\dd t} \begin{pmatrix}
	Q \\ P \\e
\end{pmatrix}
= 
\underbrace{\begin{pmatrix}
	0 & 1 & 0 \\ 
	-1& 0 & 0 \\
	0 & 0 & 0
\end{pmatrix}}_{J}
\underbrace{\begin{pmatrix}
	kQ \\ P/m \\ 1
\end{pmatrix}}_{\rmD\calE}
+
\underbrace{\frac\gamma\beta \begin{pmatrix}
	0 &    0 & 0\\
	0 &    1 & -P/m\\
	0 & -P/m & P^2/m^2\\
\end{pmatrix}}_{K}
\underbrace{\begin{pmatrix}
	0 \\ 0 \\ \beta 
\end{pmatrix}}_{\rmD \calS},
\end{equation}
with energy and entropy
\begin{equation}
	\label{eqdef:GENERIC-E-S-repeated}
	\calE(Q,P,e) := \frac k2 Q^2 + \frac{P^2}{2m} + e
	\qquad \text{and}\qquad
	\calS(Q,P,e) = \beta e.
\end{equation}

In this section we focus on how we should understand the formula `$\calS = \beta e$', and how and why this particular function drives the evolution equation. 
This understanding will proceed along the following lines:
\begin{enumerate}
	\item We start by inventing a Hamiltonian system in a larger state space with $2+2n$ variables $(Q,P,q,p)$, in such a way that $(Q,P,e)$ can be considered as a coarse-grained version of $(Q,P,q,p)$ (Section~\ref{ss:setup-micro-HamSys}).
	\item We let the dynamics of this Hamiltonian system start from  random initial data. The initial data are  drawn from a microcanonical measure that is invariant under the Hamiltonian flow (Section~\ref{ss:HamSys-random-initial-data}).
	\item We characterize the coarse-grained version of this microcanonical measure, and show how entropy appears (Section~\ref{ss:CG-CGinvmeas-entropy-ExA}).
	\item We study the coarse-grained dynamics in the limit $n\to\infty$, and argue that it converges to a noisy version of the dynamics~\eqref{eq:damped-oscillator-eq-repeated} of Example~\ref{ex:damped-harmonic-oscillator} (Section~\ref{ss:HamSys-dynamics}).
\end{enumerate}

\subsection{A larger, Hamiltonian system}
\label{ss:setup-micro-HamSys}

The \Generic system with variables $(Q,P,e)$ above will be understood by considering it as the coarse-grained version of a purely Hamiltonian system with variables $(Q,P,q,p)$. Using notation similar to that in~\eqref{eq:X-to-Y-coarse-graining} we can write that we are looking for a system with variables $(Q,P,q,p)$ and a map $\Phi_n$ such that
\begin{equation}
	\label{eqdef:CG-map-Phi_n}
	(Q,P,q,p) 
	\quad
	\xrightarrow{\quad \Phi_n\quad }
	\quad 
	(Q,P,e).
\end{equation}

We choose a setup for $(Q,P,q,p)$ of `system plus heat bath' that has been used and studied many times; see e.g.~\cite{FordKacMazur65,CaldeiraLeggett81,Zwanzig01,MielkePeletierZimmer25}. The full system is a Hamiltonian system that consists of two coupled subsystems, each a Hamiltonian system in its own right. Figure~\ref{fig:setup-HamSys} illustrates this. 
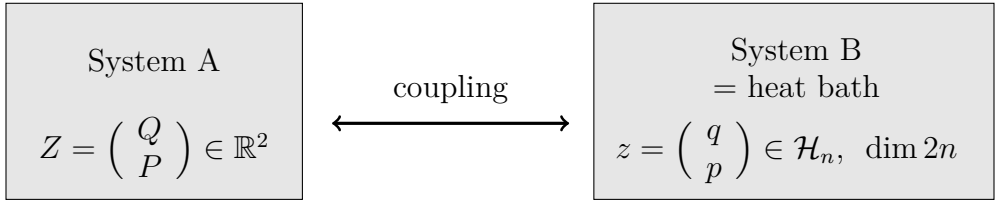
\begin{figure}[h]
	\centering
\begin{tikzpicture}[scale=0.65]
	\begin{scope}[xshift=-5.0cm] 
	\draw[fill= gray!20] (-4,0) rectangle (2,4);
	\node[color=black] at (-1,2.8){System~A};
	\node[color=black] at (-1,1){$Z = \Big(\begin{array}{c}Q\\P \end{array}\Big) \in \R^2$};
	\end{scope}
	\begin{scope}[yshift=10]
	\draw[very thick,<->] (-2.4,1.2)--(2.4,1.2);
	\node[color=black] at (0,1.9){coupling}; 
	\end{scope}
	
	\begin{scope}[xshift=5.0cm] 
	\draw[fill= gray!20] (-2.1,0) rectangle (6,4);
	\node[color=black] at (2,3){System~B};
	\node[color=black] at (2,2.3){$=$ heat bath};
	\node[color=black] at (1.8,1){$z = \Big(\begin{array}{c}q\\p \end{array}\Big)\in \calH_n$, \ $\dim 2n$};
	\end{scope}
\end{tikzpicture}
\caption{The setup of the Hamiltonian system consisting of two subsystems, `System A' and a heat bath (`System B').}
\label{fig:setup-HamSys}
\end{figure}

System A has the same variables $Z = (Q,P)$ as in Example~\ref{ex:damped-harmonic-oscillator}, with the same interpretation as position and momentum of a harmonic oscillator. The Hamiltonian for System A is
\[
H_A(Q,P) := \frac k2 Q^2 + \frac1{2m}P^2.
\]
On its own, System  A would have the evolution
\[
\dot Z = \underbrace{\begin{pmatrix}
	0 & 1  \\ 
	-1& 0  
\end{pmatrix}}_{J_A}
\underbrace{\begin{pmatrix}
	kQ \\ P/m 
\end{pmatrix}}_{\rmD H_A}
\qquad\text{i.e.}\qquad
\dot Q = \frac Pm , \quad \dot P = -kQ.
\]
Note that this evolution coincides with that of~\eqref{eq:damped-oscillator-eq-repeated} when $\gamma=0$.

System B is a separate Hamiltonian system consisting of $n$ independent harmonic oscillators. The state $z = (q,p) = (q_1,\dots,q_n,p_1,\dots,p_n)$ is an element of $\calH_n := \R^{2n}$. We equip the space $\calH_n$ with the norm 
\[
\|(q,p)\|_{\calH_n}^2 := \Delta \omega\sum_{j=1}^n \bra*{q_j^2 + p_j^2}.
\]
Note how this norm and the corresponding inner product are scaled by $\Delta\omega$. One way to interpret this expression is that the sequences $q$ and $p$ are discretizations of functions $\sfq$, $\sfp$ of frequency $\omega\in(0,\infty)$ with discretization step $\Delta \omega$, and $\|(q,p)\|_{\calH_n}^2$ is an approximation of the $L^2$-norm $\|(\sfq,\sfp)\|_{L^2(0,\infty)}^2 = \int_0^\infty \bra[\big]{\abs{\sfq(\omega)}^2 + \abs{\sfp(\omega)}^2}\dd\omega$. We choose $\Delta \omega$ to be a function of $n$ such that 
\[
\Delta \omega \longrightarrow 0 
\qquad\text{and}\qquad
n\Delta \omega \longrightarrow \infty
\qquad\text{as }n\to\infty.
\]
We do not indicate the $n$-dependence in $\Delta \omega$, to keep the notation light. 

The Hamiltonian of System  B is defined to be
\begin{equation}
	\label{eqdef:HB}
H_B(q,p) := \frac12 \Delta \omega\sum_{j=1}^n \bra*{q_j^2 + p_j^2} = \frac12 \|(q,p)\|_{\calH_n}^2,
\end{equation}
and the symplectic operator $J_B$ is chosen to be
\begin{equation}
\label{eqdef:JB-heat-bath}
J_B := \begin{pmatrix}
	&&&&\omega_1&\\
	&&&&& \ddots \\
	&&&&&& \omega_n\\
	-\omega_1 &&&\\
	& \ddots & \\
	&& -\omega_n
\end{pmatrix},
\end{equation}
with frequencies $\omega_j := j\Delta\omega>0$. 
Since\footnote{We consider $\rmD H_B$ to be the \emph{gradient} of $H_B$ in the space $\calH_n$, i.e.\ the Riesz representative of the Fr\'echet derivative. This explains why the factor $\Delta \omega$ in~\eqref{eqdef:HB} does not appear in $\rmD H_B$. If one prefers to use the usual gradient in $\R^{2n}$, then the vector $\rmD H_B$ has a factor $\Delta\omega$, and the coefficients of $J_B$ should be rescaled accordingly.} $\rmD H_B(z) = z$, the equation describing the heat-bath evolution in isolation is
\begin{equation}
\dot z = J_B z,
\qquad\text{or equivalently}\qquad
\label{eq:evol-eq-heat-bath}
\dot q_j = \omega_j p_j, \qquad \dot p_j = -\omega_j q_j, \qquad \text{for }j=1,\dots , n.
\end{equation}
One recognizes in these equations that each pair $(q_j,p_j)$ oscillates at its own frequency~$\omega_j$. 

Finally, the two subsystems are coupled in the following manner: the zero point of System B is shifted by a quantity $C_n Z$, where $C_n: \R^2 \to \calH_n$ is the operator
\begin{equation}
	\label{eqdef:C_n-choice}
C_n\begin{pmatrix} Q\\P\end{pmatrix}
:= \sqrt{\frac{2\gamma}\pi}\,Q\begin{pmatrix} 
	1\\\vdots \\ 1 \\ 0 \\ \vdots \\ 0 
\end{pmatrix}, 
\quad\text{for some $\gamma>0$, with adjoint}\quad 
C_n^*\begin{pmatrix} q\\p\end{pmatrix}
= \bra*{ \sqrt{\frac{2\gamma}\pi}\Delta \omega\sum_{j=1}^n q_j\ , \ 0 \ }.
\end{equation}
The particular form of the prefactor $\sqrt{2\gamma/\pi}$ is chosen to simplify the formulae later. 

\medskip
Combining these specifications of Systems A and B and the coupling leads to the total Hamiltonian (again writing $Z=(Q,P)$ and $z=(q,p)$ for brevity)
\begin{equation}
	\label{eqdef:Htotal}
H_{\mathrm{total}}(Z,z) := H_A(Z) + {\frac12 \|z - C_n Z\|_{\calH_n}^2}.
\end{equation}
Using the same symplectic operators $J_A$ and $J_B$ we then find the equations
\begin{subequations}
	\label{eq:Zz-coupled-system}
\begin{align}
\dot Z &= J_A \bra[\big]{\rmD H_A(Z) - C_n^* (z-C_nZ)}\label{eq:dot-Z}\\
\dot z &= J_B \bra*{z - C_n Z }.
\label{eq:dot-z}
\end{align}
\end{subequations}
The system~\eqref{eq:Zz-coupled-system} can also be written in coordinates as
\begin{alignat*}2
	\dot Q &= \frac Pm, \qquad & \dot P & = -kQ + c\,\Delta \omega \sum_{j=1}^n  \bra*{q_j -c \,Q},\\
	\dot q_j &= \omega_j p_j, \qquad & \dot p_j &= -\omega_j \bra*{q_j- c\,Q}\,,
\end{alignat*}
with $c := \sqrt{{2\gamma}/\pi}$. 

\subsection{Random initial data}
\label{ss:HamSys-random-initial-data}

We now introduce randomness, in the form of random initial data. Fix $\beta>0$, $n\in \N$, and $\calE_0\in\R$, and define the positive measure  $\mu_{\beta,n,\calE_0}$ on $\R^2\times \calH_n$ by
\begin{subequations}
	\label{eqdef:mu-beta-n-E0}
\begin{align}
\mu_{\beta,n,\calE_0}(\rmd Z\rmd z)
&:= \pdelta\pra*{H_{\mathrm{total}}(Z,z) - \frac n\beta -\calE_0}(\rmd Z\rmd z)\\
&= \pdelta\pra*{H_A(Z) + \frac12 \|z - C_n Z\|_{\calH_n}^2 - \frac n\beta -\calE_0}(\rmd Z\rmd z).
\end{align}
\end{subequations}
The measure $\mu_{\beta,n,\calE_0}$ is the microcanonical measure concentrated on the energy level $n/\beta  + \calE_0$.
This choice of energy level ensures that as $n\to\infty$ the total energy of the full system scales as $n/\beta + O(1)$. Since for large $n$ almost all the energy will be stored in the heat bath, this implies that each of the $2n$ scalar heat-bath variables will have on average energy $1/2\beta$. From a statistical-physics standpoint, this is the energy per degree of freedom that corresponds to temperature $1/\beta$. 

Note that the measure $\mu_{\beta, n,\calE_0}$ is finite and non-negative, but not normalized, and this will be important in the discussion below. We define $\sfZ_{\beta,n,\calE_0}$ to be the normalization constant, i.e.
\[
\sfZ_{\beta,n,\calE_0} := \mu_{\beta, n,\calE_0}(\R^2\times \calH_n),
\]
such that $(\sfZ_{\beta,n,\calE_0})^{-1}\mu_{\beta, n,\calE_0}$ is a probability measure on $\R^2\times\calH_n$. Below we will write  statements such as ``the random variable $X$ has law $\mu$'' or ``$X\sim \mu$'' even if $\mu$ is not normalized, in which case normalization is implied.

\begin{lemma}[Invariance under Hamiltonian flow]
	\label{l:mu-beta-n-E0-is-invariant}
The measure $\mu_{\beta,n,\calE_0}$ is invariant under the flow~\eqref{eq:Zz-coupled-system}, i.e.\ if $(Z_t,z_t)$ evolve by~\eqref{eq:Zz-coupled-system}, then
\[
(Z_0,z_0) \sim \mu_{\beta,n,\calE_0} \qquad 
\Longrightarrow
\qquad \forall t\in \R: \ (Z_t,z_t) \sim \mu_{\beta,n,\calE_0}.
\]
\end{lemma}

\begin{proof}
The conservation of $\mu_{\beta,n,\calE_0}$ by the flow follows from the definition~\eqref{eqdef:pdelta} and the fact that the flow conserves both the Lebesgue measure $\rmd Z \rmd z$ and all characteristic functions of the form
\[
\bONE\{a\leq H_{\mathrm{total}}(Z,z)\leq b\}(Z,z).
\qedhere
\]
\end{proof}

With this microcanonical invariant measure, we can define a stochastic process $t\mapsto (Z_t,z_t)$. We define two versions, which differ in their initial datum for $Z$. In both cases we consider fixed values of $\beta>0$, $n\in\N$, and $\calE_0\in\R$.

\begin{definition}[Microscopic stochastic process, stationary version]
	\label{def:microscopic-process-stationary}
	At $t=0$ we sample $(Z_0,z_0)$ from $\mu_{\beta,n,\calE_0}$; the value of $(Z,z)$ at time $t$ is given by solving the deterministic Hamiltonian evolution~\eqref{eq:Zz-coupled-system} to time $t$, from starting point $(Z_0,z_0)$.
\end{definition}

For the second version we need the conditional measure of $\mu_{\beta,n,\calE_0}$, conditioned on the value of $Z$. This measure is defined by disintegrating the probability measure $(\sfZ_{\beta,n,\calE_0})^{-1}\mu_{\beta, n,\calE_0}$:
\[
\frac1{\sfZ_{\beta,n,\calE_0}}\mu_{\beta, n,\calE_0}(\rmd Z \rmd z) 
=: \wh \mu_{\beta,n,\calE_0}(\rmd Z )\, \overline\mu_{\beta,n,\calE_0}(\rmd z|Z).
\]
By construction, $\wh \mu_{\beta,n,\calE_0}$ is a probability measure on $\R^2$ and, for each $Z\in\R^2$, $\overline\mu_{\beta,n,\calE_0}(\,\cdot\, |Z)$ is a probability measure on $\calH_n$.

\begin{definition}[Microscopic stochastic process, version with prescribed $Z_0$]
	\label{def:microscopic-process-given-Z0}
	Fix $Z_0\in \R^2$. 	At $t=0$ we sample $z_0$ from the conditional stationary measure $\overline\mu_{\beta,n,\calE_0}(\,\cdot\,|Z_0)$, conditioned on the value of $Z_0$; the value of $(Z,z)$ at time $t$ is given by solving the deterministic Hamiltonian evolution~\eqref{eq:Zz-coupled-system} to time $t$, from starting point $(Z_0,z_0)$.
\end{definition}

In the first version, the process $t\mapsto (Z_t,z_t)$ is stationary, and the initial data~$Z_0$ and~$z_0$ are both random. In the second version, $Z_0$ is fixed, $z_0$ is random, and the process is not stationary in time. 

We have the following characterization of the conditional stationary measure. 
\begin{lemma}[Conditional stationary measure]
	\label{l:char-cond-stat-measure}
For given $\beta>0$, $n\in \N$, and $e\in \R$, define the following measure on $\calH_n$,
\begin{equation}
	\label{eqdef:nu-beta}
\nu_{\beta,n,e}(\rmd z) := \pdelta_{\calH_n}\pra[\Big]{\frac12 \|z\|_{\calH_n}^2 - \frac n\beta - e}(\rmd z).
\end{equation}
Fix $Z_0\in\R^2$. 
Then we have the following equivalent characterizations of the random variable~$z_0$:
\begin{equation}
	\label{char:conditional-stationary-measure}
z_0\sim \overline \mu_{\beta,n,\calE_0}(\,\cdot\,|Z_0)
\quad\Longleftrightarrow\quad 
z_0 - C_n Z_0 \sim \nu_{\beta,n,\calE_0-H_A(Z_0)}.
\end{equation}
\end{lemma}

\begin{proof}
The characterization above follows by the calculation
\begin{align*}
\MoveEqLeft\sfZ_{\beta,n,\calE_0}\int F(Z_0) G(z_0) \wh\mu_{\beta,n,\calE_0}(\rmd Z_0) \overline \mu_{\beta,n,\calE_0}(\rmd z_0|Z_0)=\\
&= \int F(Z_0) G(z_0) \mu_{\beta,n,\calE_0}(\rmd Z_0\rmd z_0)\\
&=\int F(Z_0) G(z_0) \pdelta\pra*{H_A(Z_0) + \frac12 \|z_0 - C_n Z_0\|_{\calH_n}^2 - \frac n\beta -\calE_0}(\rmd Z_0\rmd z_0)\\
&=\int F(Z_0) G(\overline z_0 + C_n Z_0) \pdelta\pra*{H_A(Z_0) + \frac12 \|\overline z_0 \|_{\calH_n}^2 - \frac n\beta -\calE_0}(\rmd Z_0\rmd \overline z_0)\\
&\leftstackrel{\eqref{eq:product-rule-for-deltas}}= \int \rmd e \int F(Z_0) G(\overline z_0 + C_n Z_0) 
\pdelta\pra*{H_A(Z_0) -e}(\rmd Z_0)
\pdelta\pra*{ \frac12 \|\overline z_0 \|_{\calH_n}^2 - \frac n\beta -\calE_0 + e}(\rmd \overline z_0)\\
&= \int \rmd e  \int F(Z_0) G(\overline z_0 + C_n Z_0) 
\pdelta\pra*{H_A(Z_0) -e}(\rmd Z_0)\\
&\hspace*{7cm} \times \, 
\pdelta\pra*{ \frac12 \|\overline z_0 \|_{\calH_n}^2 - \frac n\beta -\calE_0 + H_A(Z_0)}(\rmd \overline z_0)\\
&\leftstackrel{\eqref{eqdef:pdelta-LRS}}= \int F(Z_0) G(\overline z_0 + C_n Z_0) 
\pdelta\pra*{ \frac12 \|\overline z_0 \|_{\calH_n}^2 - \frac n\beta -\calE_0 + H_A(Z_0)}(\rmd \overline z_0)\dd Z_0\\
&= \int F(Z_0) G(\overline z_0 + C_n Z_0) \,\nu_{\beta,n,\calE_0-H_A(Z_0)}(\rmd\overline z_0)\dd Z_0.
\end{align*}
Comparing the first and last line and removing the integral over $Z_0$ we find~\eqref{char:conditional-stationary-measure}. 

Note that Lemma~\ref{l:delta-times-delta} gives the identity marked~\eqref{eq:product-rule-for-deltas} only for almost all $\calE_0$; but since both $\mu_{\beta,n,\calE_0}$ and $\nu_{\beta,n,\calE_0-H_A(Z_0)}$ depend continuously on $\calE_0$ in the narrow topology, the identity transfers to all $\calE_0$.
\end{proof}

\subsection{Coarse-graining map, coarse-grained invariant measure, and entropy}
\label{ss:CG-CGinvmeas-entropy-ExA}

We now construct the coarse-graining map~$\Phi_n$ in~\eqref{eqdef:CG-map-Phi_n}. Comparing the expressions for the total Hamiltonian $H_{\mathrm{total}}$ in~\eqref{eqdef:Htotal} and the conserved \Generic energy~$\calE$ in~\eqref{eqdef:GENERIC-E-S-repeated} suggests to consider $e$ a placeholder for $\frac12 \|z - C_n Z\|_{\calH_n}^2$, up to a normalization constant. We follow this suggestion and define the coarse-graining map
\begin{equation}
	\label{eqdef:Phin-damped-oscillator}
\Phi_n: \R^2\ti \calH_n \longrightarrow \R^2\ti\R, 
\qquad 
(Z,z) \longmapsto \bra*{\,Z\,,\, e = \frac12 \|z - C_n Z\|_{\calH_n}^2 - \frac n\beta}.
\end{equation}
Note that we renormalize $e$ by subtracting a constant. This follows from the energy scaling that is implicit in the definition~\eqref{eqdef:mu-beta-n-E0}:  we expect the total energy of the heat bath to be roughly $n/\beta$, and by subtracting this in the definition of $e$ we expect $e$ to remain bounded as $n\to\infty$.

Note that with this choice we have 
\begin{equation}
\label{eq:H-total-to-E}
H_{\mathrm{total}}(Z,z) \stackrel{\eqref{eqdef:Htotal}}=
H_A(Z) +  \frac12 \|z - C_n Z\|_{\calH_n}^2 
= \underbrace{H_A(Z) + e}_{=:\ \calE(Z,e)} \;+ \;\frac n\beta.
\end{equation}
This implies that level sets $\{(Z,z): H_{\mathrm{total}}(Z,z)=c\}$ are mapped by $\Phi_n$ to level sets $\{(Z,e): \calE(Z,e)=c'\}$, and in particular that if we first sample $(Z,z)$ from the microcanonical measure $\mu_{\beta,n,\calE_0}$ and then apply the coarse-graining map $\Phi_n$, we have $\calE(\Phi_n(Z,z)) = \calE_0$.

\bigskip

As in the first option of the Basic Answer of Section~\ref{ss:basic-answer} we claim that entropy can be recognized in the density of the coarse-grained version of the invariant measure $\mu_{\beta,n,\calE_0}$. The following lemma describes this coarse-grained measure and makes the  entropy appear in this way. 

\begin{lemma}[Coarse-grained invariant measure]
	\label{l:CG-inv-meas}
	Fix $\beta>0$ and $\calE_0$. 
The measure $\mu_{\beta,n,\calE_0}$ transforms under the map $\Phi_n$ into
\begin{equation}
	\label{eq:CG-entropy-n-beta}
(\Phi_n)_\# \mu_{\beta,n,\calE_0}(\rmd Z \rmd e)
= \ee^{\calS_{n,\beta}(e)} \pdelta\pra[\big]{H_A(Z) + e - \calE_0}(\rmd Z) \dd e.
\end{equation}
Here
\begin{equation}
	\label{eqdef:Snbeta}
\calS_{n,\beta}(e) := \log \int_{\calH_n} \pdelta\pra[\Big]{\frac12 \|z\|_{\calH_n}^2 - \frac n\beta - e}(\rmd z).
\end{equation}
In addition, there exists $C(n,\beta)$ such that as $n\to\infty$ we have for each $e\in \R$
\begin{equation}
	\label{eq:l:CG-inv-meas-conv-Snbeta}
\calS_{n,\beta}(e) - C(n,\beta) \longrightarrow \beta e.
\end{equation}
\end{lemma}

The expression~\eqref{eq:CG-entropy-n-beta} gives a characterization of the coarse-grained version of $\mu_{\beta,n,\calE_0}$ in two parts:
\begin{itemize}
\item The reference measure  $\pdelta\pra[\big]{H_A(Z) + e - \calE_0}(\rmd Z \rmd e)$ is another microcanonical measure, but now for the \Generic evolution~\eqref{eq:damped-oscillator-eq-repeated}. It is concentrated on the level set of $\calE = \calE(Z,e)$ with value $\calE_0$. By an argument similar to that of Lemma~\ref{l:mu-beta-n-E0-is-invariant} one can verify that this measure is preserved by the \emph{reversible} part (the term $JD\calE$) of the \Generic evolution~\eqref{eq:damped-oscillator-eq-repeated}.  
\item The density $\exp\bra[\big]{\calS_{n,\beta}(e)}$ with respect to this measure reflects the variation in degeneracy of the coarse-graining map $\Phi_n$, similar to the discussion in Section~\ref{s:core-example}.
\end{itemize}
We comment in more detail on the structure of \Generic SDEs in Appendix~\ref{s:generic}; see also~\cite{PeletierSeri25TR}.

\medskip

Lemma~\ref{l:CG-inv-meas} suggests the following
\begin{definition}[Entropy for the damped oscillator]
We define the entropy $\calS$ as the limit expression in~\eqref{eq:l:CG-inv-meas-conv-Snbeta},
\begin{equation}
	\label{eqdef:S-damped-oscillator}
\calS(e) := \beta e.
\end{equation}
\end{definition}
The convergence result~\eqref{eq:l:CG-inv-meas-conv-Snbeta} finally shows how the simple formula `$\calS = \beta e$' appears in the limit $n\to\infty$.
Note that the $e$-independent constant in~\eqref{eq:l:CG-inv-meas-conv-Snbeta} is not important, since in the context of \Generic one can add constants to an entropy without changing the structure. 

Also note that the entropy in this system is a property of the heat bath, not of the oscillator itself; it also depends on the parameter $\beta$ that characterizes the energy per degree of freedom in the heat bath.

\begin{proof}[Proof of Lemma~\ref{l:CG-inv-meas}]
	We follow the ideas of~\cite[Sec.~4.6]{MielkePeletierZimmer25}.
We prove the identity~\eqref{eq:CG-entropy-n-beta} by remarking that if we define the auxiliary map 
\[
T_n: \R^2\ti \calH_n \longrightarrow \R^2\ti\calH_n\ti\R, 
\qquad 
(Z,z) \longmapsto \bra*{\,Z\,,\, z\, , \, e = \frac12 \|z - C_n Z\|_{\calH_n}^2 - \frac n\beta},
\]
then for any measure $\nu$ on $\R^2\ti\calH_n$ we have
\begin{equation}
	\label{eq:relation-Phi-Psi}
(\Phi_n)_\# \nu (\rmd Z\rmd e) = \int_{z\in\calH_n} (T_n)_\# \nu (\rmd Z\rmd z\rmd e).
\end{equation}
Lemma~\ref{l:delta-times-delta} in the Appendix implies that 
\begin{align*}
(T_n)_\# \mu_{\beta,n,\calE_0} (\rmd Z\rmd z\rmd e)
&= (T_n)_\# \pdelta\pra[\Big]{H_A(Z) +  \frac12 \|z - C_n Z\|_{\calH_n}^2 - \frac n\beta -\calE_0}(\rmd Z\rmd z\rmd e) \\
&= 
\pdelta\pra[\Big]{\frac12 \|z - C_n Z\|_{\calH_n}^2 - \frac n\beta -e}(\rmd z)
\,
\pdelta\pra[\big]{H_A(Z) + e -\calE_0}(\rmd Z)
\dd e.
\end{align*}
Applying~\eqref{eq:relation-Phi-Psi} we find
\begin{align*}
\MoveEqLeft(\Phi_n)_\# \mu_{\beta,n,\calE_0} (\rmd Z\rmd e)\\[2\jot]
&= \bra[\bigg]{\,
	\int_{z\in\calH_n}
\pdelta\pra[\Big]{\frac12 \|z - C_n Z\|_{\calH_n}^2 - \frac n\beta -e}(\rmd z)}
\,
\pdelta\pra[\big]{H_A(Z) + e -\calE_0}(\rmd Z)
\,
\dd e\\[3\jot]
&\leftstackrel{\eqref{eqdef:Snbeta}}= \ee^{\calS_{n,\beta}(e) }\pdelta\pra[\big]{H_A(Z) + e -\calE_0}(\rmd Z)\,\dd e.
\end{align*}
In the final identity we used the fact that in the definition of $\calS_{n,\beta}$ we can shift the origin to $C_n Z$ without changing the value of the integral. 

To show the convergence~\eqref{eq:l:CG-inv-meas-conv-Snbeta}, remark that $\pdelta\pra{\frac12 \|z\|_{\calH_n}^2 - \frac n\beta - e}$ is supported on a sphere in $2n$ dimensions of standard Euclidean radius $r^2=\Delta \omega^{-1} \bra*{2n/\beta+2e}$. Writing~$\sigma_{2n-1}$ for the area of the unit sphere in $2n$ dimensions, and using the co-area formula~\eqref{eq:pdelta-coarea}  we find
\begin{align*}
\int_{\calH_n} \pdelta\pra[\Big]{\frac12\|z\|_{\calH_n}^2 - \frac n{\beta} - e}(\rmd z)
\ =\ \  \frac1{r{\Delta\omega}} \; \underbrace{\sigma_{2n-1} r^{2n-1}}_{\substack{\text{area of sphere}\\\text{ of radius $r$}}}
= \ \frac{\sigma_{2n-1}}{\Delta\omega^{n}} \bra*{\frac{2n}{\beta}}^{n-1} \bra*{1 + \frac{\beta e}n}^{n-1}
\end{align*}
so that 
\begin{align*}
\calS_{n,\beta}(e)&=
\underbrace{\log \frac{\sigma_{2n-1}}{\Delta\omega^{n}} \bra*{\frac{2n}{\beta}}^{n-1}}_{=:\, C(\beta,n)} \  + \ (n-1)\log \bra*{1 + \frac{\beta e}n},
\end{align*}
and the convergence~\eqref{eq:l:CG-inv-meas-conv-Snbeta} follows.
\end{proof}

\subsection{Dynamics and coarse-graining}
\label{ss:HamSys-dynamics}

In the previous section we identified, in the limit $n\to\infty$, the expression $\calS = \beta e$ as an `entropy' in the sense of the Basic Answer.
We now return to the dynamics of the microscopic system, and how it is transformed by the coarse-graining map $\Phi_n$.

In Section~\ref{ss:setup-micro-HamSys} we derived the ordinary differential equations~\eqref{eq:Zz-coupled-system} for the full Hamiltonian system. After applying the coarse-graining map $\Phi_n$ in~\eqref{eqdef:Phin-damped-oscillator} we find the equations for the coarse-grained variables $(Z,e)$ with $e = \frac12 \|z - C_n Z\|_{\calH_n}^2 - \frac n\beta$:
\begin{subequations}
\label{eq:dot-Z-e-CG}
\begin{align}
\dot Z &= J_A \bra[\big]{\rmD H_A(Z) - C_n^* (z-C_nZ)}\label{eq:dot-Z-CG}\\
	\label{eq:dot-e-CG}
\dot e &= -\bra[\Big]{J_A \rmD H_A (Z), \, C_n^*(z-C_n Z)}_{\R^2}.
\end{align}
\end{subequations}
Note that this set of equations is not closed, since the right-hand side depends on~$z$, which is not fully determined by $Z$ and $e$.

Inspecting the two equations~\eqref{eq:dot-Z-CG} and~\eqref{eq:dot-e-CG} for the two coarse-grained quantities~$Z$ and~$e$ we observe that both only depend on $z$ through the combination $C_n^*(z-C_n Z)$. Using Duhamel's principle we derive the following expression for this object, in terms of the initial data $(Z_0,z_0)$ and the time-dependent function $t\mapsto Z_t$:
\begin{align}
C_n^*(z_t-C_n Z_t) 
&= Y_n(t) - \int_0^t \kappa_n(t-s) \dot Z_s \dd s,
\label{eq:Cz-CCz-RHS}
\end{align}
where we defined
\begin{equation}
	\label{eqdef:Yn-kappan}
Y_n(t) := C_n^* \ee^{tJ_B}z_0- \kappa_n(t) Z_0
\qquad \text{and}\qquad
\kappa_n(t) := C_n^*\ee^{tJ_B}C_n.
\end{equation}

\bigskip
We now consider the first of the two types of dynamics, given by Definition~\ref{def:microscopic-process-stationary}, in which we fix $\calE_0\in \R$ and sample the initial data $(Z_0,z_0)$ from the measure $\mu_{\beta,n,\calE_0}$ given in~\eqref{eqdef:mu-beta-n-E0}. Then $Y_n$ is a stochastic process in $\R^2$, and $\kappa_n$ is a deterministic time-dependent $2\ti2$ matrix. 

To study this process, we note that upon setting $\zeta_0 :=z_0 - C_nZ_0$ we find that  
\begin{equation}
	\label{eq:char-Y_n-zeta_0}
Y_n(t) = C_n^* \ee^{tJ_B}\zeta_0
\end{equation}
and  Lemma~\ref{l:char-cond-stat-measure} implies that conditioned on $Z_0$, $\zeta_0$ has law $\nu_{\beta,n,\calE_0-H_A(Z_0)}$. 


We now consider the limit $n\to\infty$.
\begin{lemma}
	\label{l:conv-HamSys}
As $n\to\infty$, assume that $\Delta \omega\to0$ and $n\Delta \omega \to \infty$. Then
\begin{itemize}
	\item the matrix-valued function $\kappa_n$ converges on $[0,\infty)$ to a matrix-valued delta function $\tilde \gamma \delta_0(\rmd t)$, where $\tilde \gamma = \begin{pmatrix} \gamma & 0\\ 0 & 0 \end{pmatrix}$;
	\item $ Y_n$ converges to white noise with intensity $\sqrt{\dfrac{2 \gamma}{\beta}} \begin{pmatrix}1\\0\end{pmatrix}$.
\end{itemize}
\end{lemma}
\noindent
This is a fairly standard result; similar constructions are in e.g.~\cite[Sec.~3]{KupfermanStuartTerryTupper02} or \cite[Sec.~5.1]{PavliotisStuart08}. We give details and prove it in Appendix~\ref{app:proof-of-convergence}.

\medskip

With this convergence, the right-hand side in~\eqref{eq:Cz-CCz-RHS} converges to 
\begin{equation}
	\label{eq:limit-of-CzZ}
  \sqrt{\frac{2 \gamma}\beta}\begin{pmatrix} 1\\0\end{pmatrix} \dd W_t - \begin{pmatrix} \gamma & 0 \\ 0 & 0\end{pmatrix} \dd Z_t,
 \qquad\text{for }t>0,
\end{equation}
where $W$ is a one-dimensional Brownian motion.  
Formally combining~\eqref{eq:dot-Z-e-CG} and~\eqref{eq:Cz-CCz-RHS} with the limit $n\to\infty$ and the expression~\eqref{eq:limit-of-CzZ} we then find the Stratonovich limit SDE 
\begin{subequations}
	\label{eq:CG-QPe-Strat-SDE}
\begin{align}
\rmd Q_t &= \frac1m P_t \dd t \\
\rmd P_t &= \bra*{-k Q_t - \frac\gamma m P_t } \dd t + \sqrt{\frac{2 \gamma}\beta}\circ \rmd W_t\\
\rmd e_t &= \gamma \frac {P_t^2}{m^2}  \dd t - \frac {P_t}{m} \sqrt{\frac{2 \gamma}\beta}\circ \rmd W_t,
\label{eq:CG-QPe-Strat-SDE-e}
\end{align}
\end{subequations}
which has the equivalent It\^o formulation (see e.g.~\cite[Ch.~6]{Evans06}; note the additional term in the equation for $e$)
\begin{subequations}
	\label{eq:CG-QPe-SDE}
\begin{align}
\rmd Q_t &= \frac1m P_t \dd t \\
\rmd P_t &= \bra*{-k Q_t - \frac\gamma m P_t } \dd t + \sqrt{\frac{2 \gamma}\beta}\dd W_t\\
\rmd e_t &= \frac\gamma m\bra[\Big]{ \frac {P_t^2}{m}  - \frac1{\beta }} \dd t - \frac {P_t}{m} \sqrt{\frac{2 \gamma}\beta}\dd W_t.
\label{eq:CG-QPe-SDE-e}
\end{align}
\end{subequations}

These equations can also be written in the form of a `\Generic SDE' (see~\eqref{eq:GSDE-def}): 
\begin{multline}
	\label{eq:damped-oscillator-generic-sde}
\rmd \begin{pmatrix}
	Q \\ P \\e
\end{pmatrix}
= 
\underbrace{\begin{pmatrix}
	0 & 1 & 0 \\ 
	-1& 0 & 0 \\
	0 & 0 & 0
\end{pmatrix}}_{J}
\underbrace{\begin{pmatrix}
	kQ \\ P/m \\ 1
\end{pmatrix}}_{\rmD\calE}
\dd t
+
\underbrace{\frac\gamma\beta \begin{pmatrix}
	0 &    0 & 0\\
	0 &    1 & -P/m\\
	0 & -P/m & P^2/m^2\\
\end{pmatrix}}_{K}
\underbrace{\begin{pmatrix}
	0 \\ 0 \\ \beta
\end{pmatrix}}_{\rmD \calS} \dd t
\\[3\jot]
+ \underbrace{\frac\gamma\beta \begin{pmatrix} 0 \\ 0 \\ -1/m\end{pmatrix}}_{\div K}\dd t 
+ \sqrt2
  \underbrace{\sqrt\frac\gamma \beta
  \begin{pmatrix} 0 \\ 1 \\ -P/m \end{pmatrix}}_{\Sigma}
  \dd W_t.
\end{multline}
where $J$, $\calE$, $K$, and $\calS$ are as in~\eqref{eq:damped-oscillator-eq-repeated} and~\eqref{eqdef:GENERIC-E-S-repeated}, the density $a$ in~\eqref{eq:GSDE-def} is the constant $1$,  and the noise mobility $\Sigma$ satisfies the fluctuation-dissipation relation $\Sigma\Sigma^\top = K$. 

Note how the first two drift terms of the SDEs~\eqref{eq:CG-QPe-SDE} and~\eqref{eq:damped-oscillator-generic-sde} correspond to the equations~\eqref{eq:damped-oscillator-eq-repeated}, while the $\beta$-dependent terms are new. We discuss the \Generic-SDE structure in detail in Appendix~\ref{s:generic}.

\begin{remark}[Rigorous derivation of the limit]
	While in this paper the convergence in Lemma~\ref{l:conv-HamSys} is proved rigorously, I only conjecture the resulting convergence of the full system to~\eqref{eq:CG-QPe-Strat-SDE}. The appearance in the convolution integral of the time derivative $\dot Z$ needs careful attention. I intend to return to this issue in a future publication.
\end{remark}

\begin{remark}[Origin of $\Sigma$ and $K$]
	\label{rem:Origin-of-Sigma-and-K}
The structure of $\Sigma $ and $K$ in~\eqref{eq:damped-oscillator-generic-sde} can be understood from the convergence of~\eqref{eq:dot-Z-e-CG} to~\eqref{eq:CG-QPe-Strat-SDE}, if one takes into account that the noisy part $Y_n$ of $C_n^* (z-C_nZ)$ converges to white noise with intensity $\sqrt{2\gamma/\beta} \,(1,0)^\top$. Writing~\eqref{eq:dot-Z-e-CG} as 
\begin{align*}
\dd Z &=  -J_A Y_n \dd t + \textrm{other terms}\\
\dd e &= -(J_A \rmD H_A(Z), Y_n)_{\R^2}\dd t + \textrm{other terms}
\end{align*}
we observe that the limit of this pair of equations can be written as 
\begin{align*}
\dd \begin{pmatrix}
	Q \\ P \\ e
\end{pmatrix}
&= \begin{pmatrix}
0 & -1 \\ 1 & 0 \\ -P/m & kQ \end{pmatrix}
\sqrt{\frac {2\gamma}{\beta}} 
\begin{pmatrix} 1 \\ 0 \end{pmatrix} \cdot \textrm{white noise}
\quad + \textrm{other terms}\\
&= \underbrace{\sqrt {\frac {2\gamma}{\beta}}
\begin{pmatrix}
0 \\ 1  \\ -P/m  \end{pmatrix}}_{\sqrt 2 \, \Sigma}
\cdot \textrm{ white noise}
\quad + \textrm{other terms}.
\end{align*}
The intensity of the white noise above is $\sqrt 2 \,\Sigma$, as in~\eqref{eq:damped-oscillator-generic-sde}. From the fluctuation-dissipation relation $\Sigma\Sigma^\top = K$ then follows the form of $K$.
\end{remark}

\begin{remark}[Origin of $\calE$ and $J$]
	\label{rem:Origin-of-E-J}
By following the convergence of~\eqref{eq:dot-Z-e-CG} to~\eqref{eq:CG-QPe-SDE} we can also recognize how the \Generic components $J$ and $\calE$ arose from the microscopic evolution in $Z$ and $z$. By the observation~\eqref{eq:H-total-to-E}, $\calE$ has the interpretation of the total energy of the microscopic system, up to a renormalization by the subtraction of $n/\beta$ to keep the value finite in the limit $n\to\infty$. This also explains why it should be conserved in the coarse-grained evolution. 

The Poisson operator $J$ in~\eqref{eq:damped-oscillator-generic-sde} is a remapping of the macroscopic operator $J_A$ to the variables $(Z,e)$, as can also be recognized in the expression~\eqref{eq:dot-Z-CG}. This also is consistent with e.g.\ \"Ottinger's prescription for the coarse-grained Poisson operator~\cite[Eq.~(6.68)]{Oettinger05}.
\end{remark}

\begin{remark}[Deterministic initial datum $Z_0$]
In the setup above we assumed that the full initial data $(Z_0,z_0)$ is a random sample from $\mu_{\beta,n,\calE_0}$, as described in Definition~\ref{def:microscopic-process-stationary}, and we heuristically obtained convergence as $n\to\infty$ to the \Generic SDE~\eqref{eq:damped-oscillator-generic-sde}. If instead we fix $Z_0\in \R^2$ and choose $z_0$ from the conditional stationary measure, as described by Definition~\ref{def:microscopic-process-given-Z0}, then the same convergence takes place; the only difference is that the initial data for $Q$ and $P$ are non-random and given by the value of $Z_0 = (Q_0,P_0)$.
\end{remark}

\subsection{Conclusions about this example}

From the various discussions in Section~\ref{s:damped-pendulum} we draw a few conclusions.

\paragraph{The \Generic system~\eqref{eq:damped-oscillator-eq-repeated} can be seen as a coarse-grained version of a Hamiltonian system.} By replacing the variable $e$ in $\calE$ by a high-dimensional heat bath, with an appropriate coarse-graining operator, we can see the \Generic system~\eqref{eq:damped-oscillator-eq-repeated} as the limit of a coarse-grained Hamiltonian system with random initial data. 

\paragraph{The natural coarse-grained evolution is random, not deterministic.} In first instance this coarse-graining leads to a \emph{random} evolution, as illustrated in~\eqref{eq:CG-QPe-SDE}. This observation prompts us to think about versions of \Generic that are random evolutions rather than deterministic ones; in Appendix~\ref{s:generic} we describe such \Generic SDEs in detail.

\paragraph{Entropy arises as density of the coarse-grained invariant measure.} The expression~\eqref{eq:CG-entropy-n-beta} clearly shows this in combination with the convergence~\eqref{eq:l:CG-inv-meas-conv-Snbeta}.

\paragraph{$\calS$ and $K$ canonically depend on $\beta$.}
In Example~\ref{ex:damped-harmonic-oscillator} in Section~\ref{se:central_question} the definitions of both $\calS$ and $K$ contain $\beta$, despite $\beta$ not appearing in the equation~\eqref{eq:ODE-harmonic-oscillator}. Indeed, in the structure~\eqref{eq:Generic-ODE-damped-harmonic-oscillator} one could redefine $\calS$ and $K$ without the factors $\beta$,  and the resulting structure would both represent~\eqref{eq:ODE-harmonic-oscillator} and be of \Generic form.

We can understand the role of $\beta$ from the combination of Sections~\ref{s:core-example} and~\ref{s:damped-pendulum}. The core example in Section~\ref{s:core-example} shows how it is natural that $K$ has a dependence $1/\beta$ on the size of the noise. The limit~\eqref{eq:l:CG-inv-meas-conv-Snbeta} shows how $\calS_{n,\beta}$ and its limit characterize the dependence on $e$ of the degeneracy of macrostates in the heat bath, namely as $\beta e$. 

This last remark is an important one. In a heat bath of temperature $1/\beta$, the size of macrostates really does depend on~$\beta$; if $\beta$ is larger, then an infinitesimal increment in $e$ leads to a larger change in degeneracy than in if $\beta$ is smaller. This is related to the fact that the constant $\beta$ characterizes the amount of energy per degree of freedom in the heat bath (see~\eqref{eqdef:mu-beta-n-E0}), and therefore is a property of the heat bath; at the same time  $\calS$ also is a property of the heat bath.

Incidentally, other treatments of an oscillator as a \Generic system have the same dependence of $K$ and $\calS$ on temperature (e.g.~\cite[Sec.~5]{Ottinger18}; see~\cite{JungelStefanelliTrussardi21} for an example with a finite heat bath).

\medskip

\paragraph{Why is \emph{friction} driven by \emph{entropy}?}
In Section~\ref{s:damped-pendulum} we saw how the forces between the system $Z$ and the heat-bath variables $z$ are mediated by the expression $C_n^*(z-C_nZ)$, which converges to the biased white noise of~\eqref{eq:limit-of-CzZ}. The bias in this noise is generated by the degeneracy of constant-energy level sets in the heat bath, i.e.\ by the entropy.

To see this, let us formally write~\eqref{eq:dot-Z-CG} with the replacement~\eqref{eq:limit-of-CzZ} as 
\[
\dot Z = J_A \bra[\Big]{\rmD H_A (Z) 
  + {\textstyle \begin{psmallmatrix} \gamma & 0 \\ 0 & 0 \end{psmallmatrix}}\dot Z + \text{noise }},
\]
or after combining the two terms with $\dot Z$, 
\[
\bra[\Big]{I - J_A \begin{psmallmatrix} \gamma & 0 \\ 0 & 0 \end{psmallmatrix}}\dot Z = J_A \bra[\Big]{\rmD H_A (Z) 
   + \text{noise }}.
\]
Using $(I- J_A \begin{psmallmatrix} \gamma & 0 \\ 0 & 0 \end{psmallmatrix})^{-1}J_A = J_A - \begin{psmallmatrix} 0 & 0 \\ 0 & \gamma \end{psmallmatrix}$, we rewrite this as 
\[
\dot Z = \bra[\Big]{J_A - \begin{psmallmatrix} 0 & 0 \\ 0 & \gamma \end{psmallmatrix}} \bra[\big]{\rmD H_A(Z) + \text{noise }}.
\]
This implies that the energy $H_A(Z)$ of the system satisfies
\[
\frac{\rmd}{\rmd t} H_A(Z) =  \underbrace{-\rmD H_A(Z)^\top \begin{psmallmatrix} 0 & 0 \\ 0 & \gamma \end{psmallmatrix} \rmD H_A(Z)}_{\leq 0} {}+ \text{noise }.
\]
This shows that $H_A$ has a drift towards lower values. Since at finite $n$ the total energy is fixed, this implies a drift towards higher energy $e$ of the heat bath, and since degeneracy in the heat bath is proportional to $e$ (see~\eqref{eqdef:S-damped-oscillator}) this leads to higher degeneracy.

The answer to the question `why is friction driven by entropy' is therefore that (by the argument above) the random motion of the heat bath favours directions in which $H_A$ decreases, and this translates into `friction'.

\section{Conclusion and discussion}
\label{s:conclusion}

Let us summarize the answers to the questions that we have found. 

\medskip
We have seen that entropy drives evolutions because it characterizes the invariant measure after coarse-graining of an underlying process; this is the most fundamental version of the answer to Question~Q0. 

\medskip

We have also seen the two forms that this driving can take. In the first form, the coarse-grained equations are stochastic and can be written e.g.\ as an SDE for a finite-dimensional stochastic process. In this SDE the derivative of the entropy appears as a drift term. 
We have seen two examples of this:
\begin{itemize}
\item In the core example of Section~\ref{s:core-example}, 
the evolution of the microscopic process $X_t$ is random and given by an SDE.\@ The coarse-graining map $\Phi_n$ maps the invariant measure~$\mu_n$ of the process $X_t$ to the measure $\ee^{\calS_n(y)}\rmd y$ of the process $Y_t$. The derivative $\calS_n'(y)$ appears in the drift in the resulting SDE~\eqref{eq:CG-example-SDE-Y} for $Y_t$.
\item In the damped-oscillator Example~\ref{ex:damped-harmonic-oscillator} in Section~\ref{s:damped-pendulum} the evolution of the microscopic variables $(Z,z)$ is deterministic, but the initial datum is random. The coarse-graining map~$\Phi_n$ in~\eqref{eqdef:Phin-damped-oscillator} maps the invariant measure $\mu_{\beta,n,\calE_0}$ to a measure
\[
\ee^{\calS_{n,\beta}(y)}\pdelta\pra[\big]{H_A(Q,P) +e - \calE_0}(\rmd y)
\qquad\text{ for $y=(Q,P,e)$.}
\]
The expression $\calS(Q,P,e) = \beta e$ arises as a translated version of $\calS_{n,\beta}(Q,P,e)$ in the limit $n\to\infty$, and its derivative appears in the \Generic SDE~\eqref{eq:damped-oscillator-generic-sde} in the drift term $K\rmD \calS$. 
\end{itemize}

The second form in which entropy drives evolutions is in deterministic systems, such as gradient flows and deterministic \generic systems. Here $\calS$ typically is the large-deviations rate functional of the invariant measures $\mu_n$ of a sequence of Markov processes. In the limit $n\to\infty$ these Markov processes converge to a deterministic limit, and the invariant measures $\mu_n$ become singular; $\calS$ characterizes the behaviour of $\mu_n$ in this limit. We have also seen a few examples of this:
\begin{itemize}
\item The functional $-\int \rho \log \rho - \int \rho V$ that drives the Fokker-Planck equation (Example~\ref{ex:JKO}) arises as the large-deviation rate function of independent Brownian particles. Here the coarse-graining map is the empirical-measure map (see Section~\ref{ss:relative-entropies-Sanov})
\[
(X_1,\dots X_n) \longmapsto \rho^n := \frac1n \sum_{i=1}^n \delta_{X_i}.
\]
\item Other functional expressions for the entropy arise from different stochastic processes, such as the hard-rod system, the Zero-Range Process, and the Brownian Energy Process (see Section~\ref{ss:interacting-particles}).
\item The limit $n\to\infty$ in the example of Remark~\ref{rem:deterministic-equations-from-combined-limits}  is a similar large-deviation limit, since the functional $\calS_n$ scales as $n\calS$ for some fixed $\calS$.
\end{itemize}

In both cases, the entropy \emph{drives} the evolution through the effect of noise, either at finite noise levels (the first form) or in a large-deviation limit of vanishing noise (the second form). The Onsager operator $K$ characterizes this driving and takes its form from the combination of the structure of the noise and the form of the coarse-graining map (see e.g.~Example~\ref{Ex:anisotropic-Wasserstein}).

\bigskip

We now revisit the subquestions that we posed in the introduction. 

\paragraph{Question~\ref{q:what-is-entropy-from-modelling-perspective}: What does it mean to be called an `entropy'? How should one interpret this from a modelling perspective?}
The Basic answer of Section~\ref{ss:basic-answer} gives insight into this: Entropy is a characterization of the coarse-grained version of an underlying invariant measure. Consequently, in modelling terms, one needs to understand (or choose) the microscopic dynamics in order to be able to formulate an expression for the entropy.

\paragraph{Question~\ref{q:why-entropy-in-so-many-forms}: Why does entropy come in so many different functional forms?} 
The Basic answer of Section~\ref{ss:basic-answer} also gives insight into this: If entropy is a reflection of the invariant measure of some underlying system, then the variation in such underlying systems naturally gives rise to variation in the expressions for the entropy.

\paragraph{Question \ref{subq:increasing}: Does entropy always increase along an evolution?} For many scientists, it is `evident' that entropy increases along an evolution. What we see here is that the situation is slightly more subtle:
\begin{enumerate}
\item In the second interpretation of the Basic Answer of Section~\ref{ss:basic-answer}, entropy is the large-deviation rate function of a coarse-grained invariant measure, as discussed in Section~\ref{s:entropy-large-deviations}. In Section~\ref{ss:LDP-evolutions} we saw how this function may simultaneously be the driver of a gradient flow or a \generic evolution.  In both cases this structure automatically causes the entropy to be monotonic.
\item In the first interpretation, on the other hand,  entropy is the Lebesgue density of the push-forward of a microscopic invariant measure, as e.g.\ in~\eqref{l:CG-example-S_n} or~\eqref{eqdef:inv-meas-core-example-e-Sn}, or in~\eqref{eqdef:S-damped-oscillator}. Another way of saying the same thing is that the coarse-grained evolution~$Y$ has an invariant measure of the form $\nu (\rmd y) = \ee^{\calS(y)}\dd y$ (see~\eqref{eqdef:inv-meas-core-example-e-Sn}). Therefore $\calS(Y_t)$ will continue to fluctuate over time, and along this evolution $\calS$ definitely is not monotonic.  The theory of Stochastic Thermodynamics deals with these aspects in much more detail~\cite{Sekimoto10,Seifert12}.
\end{enumerate}
Of course there is a continuous transition between the two cases above, as we discussed in Remark~\ref{rem:deterministic-equations-from-combined-limits}. In the combined limit $n,\beta\to\infty$ the non-monotonic fluctuations in~$\calS_n$ become smaller and smaller, to completely vanish when $n=\infty$. For each finite $n$, however,~$\calS_n$ has non-monotonic fluctuations, and there can be no  `second law of thermodynamics' at the level of the process $Y_t$. (From this point of view it seems strange that `violations of the second law' tend to create a stir (see e.g.~\cite{Gerstner02})).

\medskip
However, there is a twist:
\begin{enumerate}[resume]
\item Any $\phi$-relative entropy \emph{of the law $\rho_t$ of $Y_t$} with respect to $\nu$ is monotonic in time (see e.g.~\cite{Voigt81}). Therefore one can associate at least two entropies with a stochastic process such as $Y_t$ in Section~\ref{ss:setup-core-example}:
\begin{align*}
&\calS_n(y)  &&\text{defined in~\eqref{eqdef:CG-example-S}, which is not monotonic in time; }\\
&\overline \calS(\rho) := -\calH_\phi(\rho|\nu) &&\text{defined in~\eqref{eqdef:RelEnt}, which is monotonic (non-decreasing) in time. }
\end{align*}
In fact, the law $\rho_t$ is a solution of the corresponding Fokker-Planck equation, which is a deterministic \Generic equation in its own right (see e.g.~\cite[Sec.~1.2.5]{Oettinger05}).
\end{enumerate}

\paragraph{Question~\ref{q:modelling-interpretation-K}: Given that the operator $K$ characterizes {how} $\calS$ drives the evolution, what is the modelling interpretation of $K$?}
We have seen three ways in which $K$ arose in the coarse-graining, and in all three ways the operator $K$ embodies the size and type of the noise, even though the details vary.

In the Core example of Section~\ref{s:core-example}, by comparing the coarse-grained equation~\eqref{eq:CG-example-SDE-Y} for~$Y$ with the \Generic SDE described in Appendix~\ref{s:generic}, we can recognize an operator~$K$ as `multiplication with $1/\beta$', so that the drift term $(1/\beta) \calS_n'(Y_t)\, dt$ can be written as $K\calS_n'(Y_t)\, dt$.  (Note that since $Y$ is a one-dimensional variable, the operator $K$ maps $\R$ to~$\R$, and a linear operator from $\R$ to $\R$ is multiplication by a constant.) 

In Section~\ref{s:entropy-large-deviations} we saw that $K$ arises in quadratic dissipation potentials $\calR^*$ in large-deviation rate functionals. These dissipation potentials characterize the large-deviation behaviour of the noise, and this can be recognized  for instance in Example~\ref{Ex:anisotropic-Wasserstein}, where the anisotropy $\Sigma$ of the noise returns in the operator $K$~\eqref{eqdef:K-Sigma}. The connection with large-deviation theory also shows that the linearity of an operator $K$ is only a special case; in general one should consider $K$ a nonlinear operator of the form $
\xi \mapsto \partial_\xi \calR^*(z,\xi)$, for some~$\calR^*$ that is convex in $\xi$.

In Section~\ref{s:damped-pendulum} (see Remark~\ref{rem:Origin-of-Sigma-and-K}) we saw that the two sub-Systems A and B exchange energy; as the parameter $n$ tends to infinity, System B evolves faster and faster, and in the limit this energy exchange takes the form of white noise, whose origin can be traced back to the random initial datum of System B. The mobility of this noise can be traced back via the definition of $\kappa_n$ in~\eqref{eqdef:Yn-kappan} to the combination of the intrinsic evolution in System B (given by $J_B$) and the coupling between the two parts (given by $C_n$).

\paragraph{Question~\ref{q:interp-E-J}: what is the modelling interpretation of $\calE$ and $J$?}
For the specific case of Example~\ref{ex:damped-harmonic-oscillator}, we addressed this in Remark~\ref{rem:Origin-of-E-J}: $\calE$ is a renormalization of the total Hamiltonian energy of the microscopic system, and $J$ can be interpreted as a remapping of the System-A Poisson operator $J_A$ to the variables $(Z,e)$. It seems natural that this interpretation should hold much more widely: when a \Generic system derives from a Hamiltonian system by coarse-graining, the \Generic energy $\calE$ should be a renormalized version of the Hamiltonian energy. 

\paragraph{Question~\ref{q:unique-char-of-GENERIC}: Does the equation uniquely characterize $\calE$, $\calS$, $J$, and $K$?} We already saw in Example~\ref{ex:DSZ} that the answer is negative: the same evolution equation can have different gradient structures, and the same holds for \Generic structures. See~\cite[Cor.~3.3.1]{Mielke16a} for the impact of this indeterminacy on evolutionary $\Gamma$-convergence.

\bigskip


\MARK{Do we want to say anything about deterministic ergodic microdynamics? Even if only to say that I personally don't know enough about it?}

We now continue with some other remarks.

\paragraph{Other discussions of the interpretation of entropy.}
Every text on thermodynamics has some discussion of entropy and free energy, and in this section I will simply discuss a few that I think are significant. 

In Section~\ref{ss:basic-answer} we already touched on the idea that entropy characterizes the `volume of a region of microstates corresponding to a single macrostate', and sometimes this formulation is even used as a definition of entropy. A more precise version is given in~\eqref{l:CG-example-S_n-char2}, in which $\calS_n$ is characterized in terms of the total mass of the micro-canonical measure~$\pdelta$. 

Leli\`evre, Rousset, and Stoltz~\cite[Sec.~1.3]{LelievreRoussetStoltz10} give a careful motivation of their definitions. In particular, they discuss the role of the coarse-graining function $\Phi_n$ (which they call a `reaction coordinate' $\xi$), and their definition of `free energy'  is consistent with our~\eqref{l:CG-example-S_n-char2}. However, they do not deal with dynamics.

Evans~\cite{Evans01} and De Graaf~\cite{DeGraaf18TR} both address the question `what is entropy' from a strongly mathematical point of view, and both approach it probabilistically. Evans, however, is more interested in entropy as characterization of irreversibility and uncertainty than in its role in evolution equations, and De Graaf is motivated by understanding the geometric structure of thermodynamics.

Berdichevsky~\cite[(1.22)]{Berdichevsky97} focuses on defining  entropy for a Hamiltonian system, and gives a careful discussion and motivation of his choice. He requires that entropy be an `adiabatic invariant', and this requirement leads to the definition of entropy at energy level~$e$ to be the volume of the sublevel set $\{(q,p): H(q,p)\leq e\}$. This is different from the definition~\eqref{l:CG-example-S_n-char2}, which can be recognized as surface area of the level set $\{(q,p): H(q,p) = e\}$ equipped with an appropriate surface measure. Berdichevsky discusses these two definitions in his Section~2.11, and he remarks that in many cases the two definitions coincide in the limit of infinite degrees of freedom. 

\paragraph{Separation between system and heat bath.}
In this paper we considered the simplest microscopic precursor of a \Generic system, in which a finite-dimensional system is coupled to a $2n$-dimensional heat bath. In the limit $n\to\infty$ the heat bath has infinite energy and infinite heat capacity, leading to a fixed temperature $1/\beta$ of the heat bath and a heat-bath entropy of the form $\calS = \beta e$ as in~\eqref{eqdef:S-damped-oscillator}. In this setup only the bath is coarse-grained, and therefore all entropy is generated by the bath. 

However, \Generic is not limited to setups of this form; in most applications the `system' and the `heat bath' are the same object, for instance an interacting particle system, which typically is spatially extended. In the thermodynamical literature it is common to make an assumption of `local equilibrium', under which one can consider a complete heat bath to be present at each macroscopic spatial point, which equilibrates instantaneously. In this case both internal energy and entropy become local fields, depending on space and time, and one enters the realm of hydrodynamics. For extensive discussions about such systems, see e.g.~\cite{MullerRuggeri98,Oettinger05,PavelkaKlikaGrmela18}.

\paragraph{`Entropy production' and local detailed balance.}
Example~\ref{ex:damped-harmonic-oscillator} illustrates how removing a `heat bath' from a Hamiltonian system leads to a reduced system with an associated entropy---where the entropy is in fact the entropy of the heat bath.  Katz, Lebowitz, and Spohn~\cite{KatzLebowitzSpohn83,KatzLebowitzSpohn84} and later Maes, Neto\v cn\'y, and others~\cite{Maes99,MaesRedigVan-Moffaert00,Maes03,MaesNetocny03} developed the pair of concepts of \emph{local detailed balance} and \emph{entropy production} to capture  the structure of this type of open systems in much more generality, in particular including multiple `heat baths'.  In simple terms, a stochastic process satisfies local detailed balance if the logarithm of the ratio of forward to backward paths can be expressed in terms of a function  of those paths; this quantity is then called the `entropy production'.

For systems with reversible dynamics (including possibly flipping the signs of velocities) this `entropy production' reduces to the change of an actual entropy, as described in this paper; see e.g.~\cite{Maes21}. For fundamentally non-equilibrium systems, such as those driven by multiple reservoirs, there does not seem to be a clearly defined entropy. This is also illustrated by the derivation in e.g.~\cite{Maes21} from multiple-reservoir systems, in which there is one entropy for each reservoir.

\MARK{Do we want to discuss why there is no assumption of the Compression Property here? Why we can make do with very weak requirements on the frequencies and such?}

\paragraph{Acknowledgements.}
I first learned about the interpretation~\eqref{l:CG-example-S_n-char1} of entropy and free energy from Tony Leli\`evre at the Lorentz Center in Leiden, and I am grateful to both for this. I am also very grateful to Victor Berdichevsky, Nino Dekkers, Jin Feng, Klaas Landsman, Christian Maes, Micha\l\ Pavelka,  Celia Reina Romo, and Andr\'e Schlichting for animated and enlightening discussions about entropy and other things. I also am very grateful for the comments in the `Wednesday morning session' in Eindhoven by Rafael Bailo and others, as well as for comments from Alessandro Scagliotti. Finally, I want to thank the two reviewers for their extensive review and comments, which were very useful in improving the paper. 

This paper grew out of the collaboration with Alex Mielke and Johannes Zimmer on~\cite{MielkePeletierZimmer25}, and I am very grateful for all the good discussions that we had together in the course of that project. 

\paragraph{Declarations.}
The author has no competing interest to declare. No funding was received to assist with the preparation of this manuscript.

\appendix

\section{Notation}
\label{s:notation}

\subsection{Generalities}

In these notes we use $\rmD F$ for an unspecified derivative of a functional $F$ on a state space~$\bfZ$, which can be infinite-dimensional. Depending on the situation this could be either a Fr\'echet derivative or a gradient, i.e.\ the Riesz representative of that same Fr\'echet derivative. (This does not lead to inconsistencies, because all such derivatives are inputs to operators $J$ and $K$; the choice of the derivative representation forces a corresponding choice for these operators.) The symmetry and anti-symmetry properties of $J$ and $K$ are defined in terms of their behaviour as bilinear forms. We reserve the notation $\nabla$ for the usual gradients in $\R^d$.

\medskip

Recall that the pushforward $T_\# \mu$ of the measure $\mu\in \ProbMeas(\calX)$ under a map $T:\calX \to\calY$ is the measure defined by $T_\#\mu(A) = \mu(T^{-1}(A))$ for any set $A$, or equivalently by 
\begin{equation}
\label{eqdef:push-forward}
\int_\calY \varphi(y) \, T_\# \mu(\rmd y) = \int_\calX \varphi(T(x))\, \mu(\rmd x)
\qquad \text{for all measurable }\varphi: \calX\to\R.
\end{equation}
Yet again equivalently, if $\mu$ is the law of a random variable $X$,  then $T_\#\mu$ is the law of the random variable $T(X)$.

\subsection{Microcanonical measures}
\label{app:pdelta}

In Section~\ref{s:core-example} we introduced the `microcanonical' measure $\pdelta$ by the heuristic definition~\eqref{eqdef:pdelta}.
Leli\`evre, Rousset, and Stoltz give a more precise definition.
\begin{definition}[{\cite[(1.25)]{LelievreRoussetStoltz10}}]
  \label{def:pdelta}
Let $f:\R^n\to\R$ be continuous.
The measure $ \pdelta\pra*{f(x)-a}$ on $\R^n$ is defined by the property that for all $F\in C_c(\R^n)$ and $G\in C(\R)$, 
\begin{equation}
\label{eqdef:pdelta-LRS}
\int_\R G(a) \int_{\R^n} F(x) \pdelta\pra*{f(x)-a}(\rmd x) \dd a
= \int_{\R^n} G(f(x)) F(x) \dd x.
\end{equation}
\end{definition}
\noindent
The right-hand side above is finite for all $F\in C_c(\R^n)$ and $G\in C(\R)$,
since $\supp F$ is compact and $G\circ f$ therefore is bounded on this
support. This implies that $\pdelta\pra*{f(x)-a}(\rmd x) \dd a$ can be
considered a locally finite non-negative measure on $(x,a)\in \R^n\ti \R$. For
almost all $a\in \R$ the disintegration $\pdelta\pra*{f(x)-a}(\rmd x)$ is
then a locally finite non-negative measure on~$\R^n$.

When $f$ is regular and $\nabla f$ does not vanish, the co-area formula provides the alternative characterization of $\pdelta$ in terms of the Hausdorff measure as 
\begin{equation}
	\label{eq:pdelta-coarea}
\pdelta\pra[\big]{f(x) - a}(\rmd x) = \frac1{|\nabla f(x)|}\calH^{n-1}\Big|_{f^{-1}(a)}(\rmd x) \qquad \text{on }\R^n.
\end{equation}
This shows that the measure is concentrated on the level set, but not equal to the Hausdorff measure on that set; the density depends on the rate of change of $f$ normal to the level set. 

\begin{remark}
	Note that $\pdelta[f(x)-a]$ is a non-negative measure, but it has no reason to have unit mass. 
\end{remark}

The following lemma and its proof are given in~\cite[Lemma~B.2]{MielkePeletierZimmer25}. It gives a characterization of a specific type of products of $\pdelta$-measures.


\begin{lemma}[Products of $\pdelta$'s]
  \label{l:delta-times-delta}
Let $X$ and $Y$ be finite-dimensional spaces, and let $f\colon X\to\R$ and $g\colon Y\to \R$ be continuous. 
Set 
\[
T\colon X\ti Y \to X\ti Y \ti \R,
\qquad
(x,y) \mapsto (x,y,e := g(y)).
\]
For Lebesgue almost all $a\in \R$ we then have
\[
\bra[\Big]{T_\# \pdelta\pra*{f(x)+g(y)- a}} (\rmd x\dd y \dd e)
= \pdelta\pra*{f(x)-a+e}(\rmd x) \pdelta\pra*{g(y)-e}(\rmd y ) \dd e.
\]
Equivalently, for all $\varphi\in C_c(X\ti Y\ti \R)$ we have
\begin{multline}
\int_{X\ti Y} \varphi(x,y,g(y)) \pdelta\pra*{f(x)+g(y)- a} (\rmd x\rmd y)\\
= \int_{X\ti Y\ti \R} \varphi(x,y,e) \pdelta\pra*{f(x)-a+e}(\rmd x) \pdelta\pra*{g(y)-e}(\rmd y ) \dd e.
\label{eq:product-rule-for-deltas}
\end{multline}
\end{lemma}

\section{\Generic systems and \Generic SDEs}
\label{s:generic}

\MARK{Add a better description of full GENERIC, not just GSDEs}

The name \emph{\Generic} was coined by Grmela and \"Ottinger~\cite{GrmelaOttinger97,OttingerGrmela97} and refers to a class of deterministic evolution equations that contain both dissipative (gradient-flow) and conservative (Hamiltonian) components. This class generalizes the earlier concept of \emph{metriplectic} evolutions~\cite{Kaufman84,Morrison84,Morrison86}. We briefly described this class of evolution equations in the introduction, and for general information about \Generic we refer to~\cite{Oettinger05,Mielke11,PavelkaKlikaGrmela18}.

\paragraph{Deterministic \Generic.}

The deterministic version of \Generic was heuristically described in Section~\ref{se:central_question}, and can be formulated for evolutions in infinite-dimensional normed vector spaces $\bfZ$. The core components are two functionals $\calE,\calS:\bfZ\to \R$ and two duality operators $J$ and $K$. We assume that $\calE$ and $\calS$ are Fr\'echet differentiable, and we indicate the Fr\'echet derivative as $\rmD \calE$ and $\rmD\calS$. For each $z\in \bfZ$, $J(z)$ and $K(z)$ are bounded or unbounded operators $J(z),\, K(z): \bfZ^* \to \bfZ$, that are required to satisfy a number of conditions:
\begin{itemize}
\item The symmetry conditions
\begin{equation}
	\label{eq:JK-symmetry}
	J(z)^\top = -J(z)\qquad  \text{and} \qquad K(z)^\top = K(z)\geq0.
\end{equation}
Here the superscript $\top$ indicates the Banach dual, i.e.\ $J(z)^\top:\bfZ^*\to \bfZ$ is defined by the requirement that 
\[
{}_{\bfZ}\langle J(z)^\top \xi, \eta\rangle_{\bfZ^*}
= {}_{\bfZ^*}\langle \xi, J(z)\eta\rangle_{\bfZ}
\qquad\text{for all }\xi,\eta\in \dom{J(z)}.
\]
\item The \emph{non-interaction conditions}
\begin{equation}
	\label{eq:NIC}
	J(z)\rmD \calS(z) = 0 \qquad \text{and} \qquad K(z)\rmD \calE(z) = 0.
\end{equation}
\item The \emph{Jacobi identity}: The Poisson bracket generated by $J$, 
\[
\{ F,G\}(z) := {}_{\bfZ}\langle J(z)^\top \rmD F(z), \rmD G(z)\rangle_{\bfZ^*}
\]
should satisfy
\[
\{F,\{G,H\}\} + \{G,\{H,F\}\} + \{H,\{F,G\}\} = 0
\qquad\text{for all admissible } F,G,H:\bfZ\to\R.
\]
\end{itemize}
Under these conditions, the \Generic evolution equation is 
\[
	\dot z = J(z) \rmD \calE(z) + K(z) \rmD \calS(z).
\]

\paragraph{\Generic SDEs.}
Grmela and \"Ottinger also identified an extension that introduces randomness while preserving the main properties (`\generic with fluctuations'~\cite{GrmelaOttinger97,Oettinger05}). In this paper we describe a slight variation of this stochastic extension, that deals with changes of variables in a different way. To distinguish from `\generic with fluctuations' we call such extensions `\Generic SDEs'. 

For \Generic SDEs we take the state space $\bfZ$ to be $\R^d$, and we write $\nabla F$ for the usual gradient of a function $F:\R^d\to\R$. We also use matrix notation, representing elements of~$\bfZ$ as column vectors and operators as matrices. With this choice, $J(z)^\top$ is the usual transpose of the matrix $J(z)$, and $\div K(z)$ is the column vector of divergences of the rows of $K(z)$.

\Generic SDEs are characterized by the five components $\bfZ$, $\calE$, $\calS$, $J$, and $K$ that we already discussed, together with two more components, $\Sigma$, and $a$. For each $z\in \bfZ$, 
$\Sigma(z)$ is a matrix satisfying the fluctuation-dissipation relation
\begin{equation}
	\label{eq:FDR-general}
\Sigma \Sigma^\top  = K.
\end{equation}
The function $a:\bfZ\to (0,\infty)$ is a positive density, which is required to satisfy the \emph{unimodularity} constraint with $J$~\cite{Weinstein97,PeletierSeri25TR},
\begin{equation}
\label{eq:unimodularity}
\div \pra[\big]{aJ}= 0.
\end{equation}

With these components, the \Generic SDE is
\begin{equation}
	\label{eq:GSDE-def}
\rmd z_t = \bra[\Big]{J(z_t)\nabla \calE(z_t) + K(z_t) \nabla \calS(z_t) + \frac1{ a(z_t)} \div \pra[\big] {aK}(z_t)}\dd t
+ \sqrt{2} \Sigma(z_t) \dd W_t.
\end{equation}
Here $W$ is a standard Brownian motion on $\R^d$.

\medskip

\begin{remark}[Geometric formulation]
For simplicity of the exposition we have given these \Generic SDEs in $\R^d$. A better geometric description places the evolution in the context of a manifold in which the components above have well-defined geometric meaning; we describe this in~\cite{PeletierSeri25TR}.
\end{remark}

\begin{remark}[Fokker-Planck equation and stationary measure]
By applying It\^o's lemma in the usual way one finds that the law $\rho_t\in \ProbMeas(\bfZ)$ of $z_t$ satisfies the Fokker-Planck equation
\begin{subequations}
\label{eq:G-FPeq}
\begin{align}
	\partial_t \rho_t + \div \pra[\big]{\rho_t{J\nabla \calE}}
	&= \div \pra*{\rho_t K \nabla \log \frac{\rho_t}{a\ee^{\calS}}}
	\label{eq:G-FPep-1}\\
	&= \div \pra[\big]{K \nabla \rho_t} - \div\pra[\big]{\rho_t K\bra*{\nabla \log a +  \nabla \calS}}.
\end{align}
\end{subequations}

\end{remark}

\begin{remark}[Almost-sure conservation of $\calE$]
In Section~\ref{se:central_question} we already described how in a deterministic \Generic equation the non-interaction conditions~\eqref{eq:NIC} lead to conservation of $\calE$. 
The setup of the \Generic SDE~\eqref{eq:GSDE-def} implies the same conservation of $\calE$, almost surely. This can be recognized by applying It\^o's lemma:
\begin{align*}
\rmd \calE(z_t) &= \nabla\calE(z_t)^\top   \bra[\Big]{J(z_t)\nabla\calE(z_t) + K(z_t)\nabla\calS(z_t)
+  \frac1{a(z_t)} \div\pra[\big]{a\Sigma\Sigma^\top}(z_t) }\dd t \\
&\qquad {} + \sqrt{2}\nabla\calE(z_t)^\top   \Sigma(z_t) \dd W_t 
  +  \nabla^2\calE(z_t)  \,{:}\, {\Sigma\Sigma^\top(z_t)}  \dd t \\
&\leftstackrel{(*)}=  \frac1{a(z_t)}\nabla\calE(z_t)^\top \div\pra[\big]{a\Sigma\Sigma^\top}(z_t) \dd t
 +  \nabla^2\calE(z_t)  \,{:}\, {\Sigma\Sigma^\top(z_t)}  \dd t \\
&\qquad {} +  \sqrt{2}\nabla\calE(z_t)^\top    \Sigma(z_t) \dd W_t \\
&=  \frac1{a(z_t)}\div \pra[\big]{a K(z_t)\nabla\calE}(z_t) \dd t 
+ \sqrt{2}\nabla\calE(z_t)^\top   \Sigma(z_t) \dd W_t \\
&= 0.
\end{align*}
The identity $(*)$ follows from antisymmetry and the non-interaction conditions~\eqref{eq:NIC}, and the final two terms vanish by~\eqref{eq:NIC} and $|\Sigma \nabla 
\calE|^2 = \nabla \calE^\top K \nabla \calE = 0$.
\end{remark}

\begin{remark}[Stationary measures]
From the expression~\eqref{eq:G-FPep-1} one deduces that the measure $\mu(\rmd z) = \ee^{\calS(z)}a(z) \dd z$ is stationary for the flow.  In fact there are many more stationary measures: since $\calE$ is almost surely conserved by the SDE~\eqref{eq:GSDE-def}, any measure of the form $\tilde\mu(\rmd z) = f\bra[\big]{\calE(z)}\mu(\rmd z)= f\bra[\big]{\calE(z)}\ee^{\calS(z)}a(z) \dd z$ also is conserved. If one allows $f$ to be a Dirac delta, then the coarse-grained measure~\eqref{eq:CG-entropy-n-beta} also is of this form (with $a\equiv1$).

This might suggest that there also are many different choices for the entropy $\calS$. In fact it is not that simple: If one modifies the entropy from $\calS$ to say $\calS + \lambda \calE$, then this breaks the non-interaction condition~\eqref{eq:NIC-intro} (even though the equation~\eqref{eq:GENERIC} is unchanged). To which extent this multitude of invariant measures constrains the \generic structure (see Question~\ref{q:unique-char-of-GENERIC}) is topic of further research.
\end{remark}

\begin{remark}[No explicit temperature]
In the \Generic SDE~\eqref{eq:GSDE-def} the temperature parameter $\beta$ does not appear explicitly; this contrasts with many common formulations of randomly perturbed systems, in which the strength of the noise (here $\Sigma$) is parametrized by a temperature. 

Such an explicit-temperature formulation can be obtained by redefining $\calS$, and $K$ as 
\[
\wt\calS(z) := \frac1\beta \calS(z)
\qquad\text{and}\qquad
\wt K(z) := \beta K(z),
\]
leading to the equation
\[
\rmd z_t = \bra[\Big]{J(z_t)\nabla \calE(z_t) + \wt K(z_t) \nabla \wt\calS(z_t) + \frac1{ \beta a(z_t)} \div \pra[\big] {a\wt K}(z_t)}\dd t
+ \sqrt{\frac2\beta} \Sigma(z_t) \dd W_t.
\]
In terms of $\wt K$ the fluctuation-dissipation relation now reads as 
\[
\Sigma\Sigma^\top = \frac1\beta\wt K,
\]
and in this form one recognizes the parametrization of $\Sigma$ by temperature $1/\beta$. 

The explicit formulae for $\calS$ and $K$ in e.g.~\eqref{eq:Generic-ODE-damped-harmonic-oscillator} and the linear-heat-bath limit~\eqref{eq:l:CG-inv-meas-conv-Snbeta} both suggest that in this case $\wt\calS$ and $\wt K$ may be natural alternatives to $\calS$ and $K$. In more general situations, however, `temperature' is not a fixed parameter but the function $\partial \calS/\partial e$, and no such simple formula exists. See also the discussion in Section~\ref{s:conclusion}.
\end{remark}

\begin{remark}[The unimodularity condition]
	\label{rem:unimodularity}
The unimodularity condition~\eqref{eq:unimodularity} guarantees the stationarity of $e^{\calS(z)}a(z)\dd z$, and can be motivated from the fact that the Liouville measure is conserved in Hamiltonian dynamics. To understand the role in stationarity, note that the time-constant measure $\rho(\rmd z) =  e^{\calS(z)}a(z)\dd z$ annihilates the first and last terms in~\eqref{eq:G-FPep-1}, and therefore it is a stationary solution iff
\begin{equation}
	\label{eq:calc-unimodularity}
0=\div\pra[\big]{a \ee^\calS J \nabla \calE}=\ee^\calS\div\pra[\big]{aJ}\cdot \nabla \calE+a \ee^\calS\, J {:}\nabla^2 \calE+a \ee^\calS \nabla \calS^{\top} J \nabla \calE \,.
\end{equation}
The last two terms vanish by the antisymmetry of $J$ and the non-interaction condition~\eqref{eq:NIC}; the unimodularity enforces the vanishing of the first term. 

As for the motivation from Hamiltonian dynamics, there is an argument to consider that the \generic SDE evolution has two associated measures: the invariant measure $\mu_{\mathrm{inv}}(\rmd z) := e^{\calS(z)}a(z)\dd z$ and a Liouville-like measure $\mu_{\textrm L}(\rmd z) := a(z)\, \rmd z$. The invariant measure $\mu_{\textrm{inv}}$ is a coarse-grained version of an invariant measure of the underlying microscopic dynamics, as broadly discussed in this paper. The Liouville measure $\mu_{\textrm L}$ is conserved by the Hamiltonian part of the dynamics, i.e.
\[
\div\pra[\big]{ \mu_{\textrm L} J \nabla \calE}
 = \div\pra[\big]{ a J\nabla \calE} = 0, 
\]
and the calculation~\eqref{eq:calc-unimodularity} shows that this is a property of $J$ alone: under the unimodularity condition this measure is preserved for any choice of $\calE$. This measure is inherited from the Liouville measure of the underlying Hamiltonian dynamics. This can also be recognized in the fact that $a$ is the constant $1$ in~\eqref{eq:damped-oscillator-generic-sde}, which corresponds to the fact that for canonical coordinates $Q,P$ the Liouville measure is the Lebesgue measure.
\end{remark}

\section{Proof of Lemma~\ref{l:conv-HamSys}}
\label{app:proof-of-convergence}

In this section we establish Lemma~\ref{l:conv-HamSys} as a consequence of Lemma~\ref{l:conv-stoch-processes-app} below. We recall the setup. 

Fix $\beta,\gamma>0$, and let $C_n$ and $J_B$ be given as in Section~\ref{s:damped-pendulum}. For given $\zeta_{0n}\in \calH_n$ we define two $\R^2$-valued processes
\begin{equation}
	\label{eqdef:Yn-appendix}
Y_n(t) := C_n^* \ee^{tJ_B}\zeta_{0n}
\qquad \text{and}\qquad
B_n(t) := \sqrt{\frac\beta{2\gamma}}\int_0^t Y_n(s)\dd s.
\end{equation}
In addition we set 
\[
\kappa_n(t) := C_n^*\ee^{tJ_B}C_n.
\]
We also define for $I\subset \R$
\[
\Lipc(I) := \set[\Big]{\varphi\in \Lip(I): \text{ $\varphi$ has compact support }}\,.
\]

\begin{lemma}
	\label{l:conv-stoch-processes-app}
	Fix $\beta,\gamma>0$ and $e\in \R$, and assume for each $n$ that  $\zeta_{0n}$ is sampled from~$\nu_{\beta,n,e}$ in~\eqref{eqdef:nu-beta}. Then as $n\to\infty$, 
\begin{enumerate}
\item \label{l:conv-HamSys-appendix:i:Bn} $B_n$ converges in distribution in $C([0,T];\R^2)$ to a process
\[
B(t) = \begin{pmatrix}
	W(t)\\0
\end{pmatrix},
\]
where $W$ is a one-dimensional standard Brownian motion.
\item \label{l:conv-HamSys-appendix:i:Kn} 
The sequence $\kappa_n$ converges in duality with $\Lipc$  to 
\begin{align}
2\gamma \begin{pmatrix} 1&0\\0&0 \end{pmatrix}  \delta_0
 \qquad & \text{on $\R$, and }
 \label{l:conv-HamSys-appendix:eq:Kn-on-R}\\[\jot]
\gamma \begin{pmatrix} 1&0\\0&0 \end{pmatrix} \delta_0
 \qquad & \text{on $[0,\infty)$}.
 \label{l:conv-HamSys-appendix:eq:Kn-on-[0,infty)}
\end{align}
\end{enumerate}
\end{lemma}

\begin{proof}
By integrating~\eqref{eq:evol-eq-heat-bath} we find that the solution $(q,p)(t) = \ee^{tJ_B}(q^0,p^0)$ starting at some $(q^0,p^0)$ of the pure heat-bath evolution is given by
\[
q_j(t) = q_j^0 \cos \omega_j t + p_j^0 \sin \omega_j t
\qquad\text{and}\qquad 
p_j(t) = -q_j^0 \sin \omega_j t + p_j^0 \cos \omega_j t .
\]
It follows that for this same starting point $\zeta_{0n} = (q^0,p^0)$, the process $Y_n$ in~\eqref{eqdef:Yn-appendix} satisfies 
\begin{align*}
Y_n(t) &= \bra[\bigg]{\sqrt{\frac{2\gamma}{\pi}}\Delta \omega \sum_{j=1}^n \pra[\Big]{q_j^0 \cos \omega_j t + p_j^0 \sin \omega_j t},\, 0 \, },\\
\noalign{\noindent and writing $W_n$ for the first coordinate of $B_n$ we have }
W_n(t) &= \Delta \omega \sqrt{\frac\beta\pi}\sum_{j=1}^n \frac1{\omega_j}\pra[\Big]{q_j^0 \sin \omega_j t + p_j^0 (1-\cos \omega_j t)}.
\end{align*}
By similar calculations we also find for any $(Q,P)^\top\in \R^2$
\begin{equation}
	\label{eq:char-kappa_n-proof-appendix}
\kappa_n(t)\begin{pmatrix} Q\\P \end{pmatrix} := C_n^* \ee^{tJ_B} C_n \begin{pmatrix} Q\\P \end{pmatrix}
= \bra[\bigg]{Q\,\frac{2\gamma}{\pi}  \Delta \omega \sum_{j=1}^n \cos \omega_j t\;, \; 0 \; }^\top.
\end{equation}

We now first prove part~\ref{l:conv-HamSys-appendix:i:Kn} of the Lemma.
For a function $\psi\in \Lipc(\R)$ 
we have the Fourier forward and backward transform identities (where the first is a definition, the second a property)
\[
\wh\psi(b) := \int_\R \ee^{-i ab} \psi(a) \dd a
\qquad\text{and}\qquad 
\psi(a) = \frac1{2\pi}\int_\R \ee^{i ab} \wh \psi(b) \dd b.
\]
Note that $\psi\in \Lipc(\R)$ implies that $\wh\psi$ is continuous and Lebesgue integrable on $(0,\infty)$, and therefore Riemann integrable.

Fix $\varphi \in \Lipc(\R)$, and note that since $\varphi$ is real-valued we have $\wh \varphi(-\omega) = \overline{\wh \varphi(\omega)}$. We calculate
\begin{align*}
\int_\R \varphi(t) \Delta \omega \sum_{j=1}^n \cos (\omega_j t) \dd t
&= \Delta \omega \sum_{j=1}^n  \Re \wh\varphi(\omega_j)\\
&\longrightarrow \int_0^\infty \Re \wh \varphi(\omega)\dd \omega \qquad \text{as }n\to\infty\\
&= \frac12 \int_\R \wh\varphi(\omega) \dd \omega  
= \pi \varphi(0).
\end{align*}
Combining this with~\eqref{eq:char-kappa_n-proof-appendix}   proves~\eqref{l:conv-HamSys-appendix:eq:Kn-on-R}. The corresponding result~\eqref{l:conv-HamSys-appendix:eq:Kn-on-[0,infty)} on $[0,\infty)$ is proved  by noting that $\kappa_n$ is even, so that extending $\varphi\in \Lipc([0,\infty))$ by even symmetry to $\varphi\in \Lipc(\R)$ we have 
\[
\int_0^\infty \varphi(t) \kappa_n(t)\dd t
= \frac12 \int_\R \varphi(t) \kappa_n(t)\dd t
\xrightarrow{n\to\infty}
\varphi(0)\begin{pmatrix} \gamma &0\\0&0 \end{pmatrix}.
\]

\medskip

We now prove part~\ref{l:conv-HamSys-appendix:i:Bn}. By~\cite[p.~450]{GikhmanSkorokhod69} it is sufficient to show convergence in distribution of all finite-time marginals of $W_n$ (the first coordinate of $B_n$) to those of a standard Brownian motion $W$, together with control of time variation. 

To show convergence of finite-time marginals, fix $m\in \N$ and $0\leq t_0 \leq \dots \leq t_m$, and choose any $a_1,\dots,a_m\in\R$. We show that the scalar random variable
\[
A_n := \sum_{i=1}^m a_i(W_n(t_i)-W_n(t_{i-1}))
\]
converges in distribution to a centered scalar Gaussian random variable with variance 
\[
\sum_{i=1}^m a_i^2 (t_i-t_{i-1}).
\]
This implies that all finite-time  marginals of $W_n$ converge in distribution to multidimensional Gaussian random variables with the same covariance matrix as the marginals of~$W$.

To show the convergence of $A_n$, note that $A_n$ can be written as 
\[
A_n = (\zeta_{0n},\alpha_n)_{\calH_n}, 
\]
where $\zeta_{0n} = (q^0_j,p^0_j)_{j=1}^n\sim \nu_{\beta,n,e}$ and 
\[
\alpha_{n,j} = \sqrt{\frac\beta\pi } \sum_{i=1}^m \frac{a_i}{\omega_j} \bra[\Big]{(\sin\omega_j t_i - \sin\omega_j t_{i-1}),\, (-\cos \omega_j t_i + \cos \omega_j t_{i-1})},
\qquad j=1,\dots,n.
\]
By lengthy but elementary calculations, using the fact that $\Delta\omega\to0$ and $n\Delta \omega\to\infty$,  we find that
\begin{align*}
\|\alpha_n\|_{\calH_n}^2 
&\longrightarrow \beta \sum_{i=1}^m a_i^2 (t_i-t_{i-1}),
\qquad\text{as }n\to\infty.
\end{align*}
The random variable $\zeta_{0n}$ is drawn from a uniform measure on a sphere in the $2n$-dimensional space $\calH_n$ with
\[
\|\zeta_{0n}\|_{\calH_n}^2  = \frac{2n}\beta +2e.
\]
By standard equivalence-of-ensemble results (e.g.~\cite{DiaconisFreedman87}) it follows that the one-dimensional projections $(\zeta_{0n},\alpha_n)_{\calH_n}$ converge in total variation (and therefore in distribution) to a centered Gaussian random variable with variance
\[
\lim_{n\to\infty} \frac1{\dim\calH_n} \|\zeta_{0n}\|_{\calH_n}^2 \|\alpha_n\|_{\calH_n}^2
= 
 \sum_{i=1}^m a_i^2 (t_i-t_{i-1}),
\]
as claimed. 

\medskip

Finally, to show control of time variation, take any $0\leq t_0< t_1$ and set 
\[
X_n := W_n(t_1) - W_n(t_0) = (\zeta_{0n},\alpha_n)_{\calH_n},
\]
where we use the same $\alpha_n$ as above but with just two times $t_0$ and $t_1$ and with $a_1 = 1$. 
By e.g.~\cite[Th.~3.4.6]{Vershynin18}, $\zeta_{0n}$ is a sub-Gaussian random variable in $\R^{2n}$ with a bound  $\|\zeta_{0n}\|_{\psi_2}\leq C\Delta\omega^{-1/2}$ with a constant $C$ that is independent of $n$ (see~\cite{Vershynin18} for the definition of the $\psi_2$-norm). By definition of sub-Gaussian random variables, the one-dimensional reduction $(\zeta_{0n},\alpha_n)_{\calH_n} = \Delta\omega (\zeta_{0n},\alpha_n)_{\R^{2n}}$ is again sub-Gaussian, with $\psi_2$-norm bounded by 
\[
\|X_n\|_{\psi_2}
= \Delta \omega \|(\zeta_{0n},\alpha_n)_{\R^{2n}}\|_{\psi_2}
\leq \Delta\omega \|\alpha_n\|_{\R^{2n}} \|\zeta_{0n}\|_{\psi_2}
\leq C \Delta\omega^{1/2} \|\alpha_n\|_{\R^{2n}}.
\]
We estimate the right-hand as follows. Note that the expression
\[
\Delta\omega \|\alpha_n\|_{\R^{2n}}^2 
= \frac{2\beta}\pi \Delta \omega \sum_{j=1}^n \frac1{\omega_j^2} (1-\cos\omega_j t)\, \qquad t := t_1-t_0
\]
is a Riemann sum for the integral
\[
\frac{2\beta}\pi \int_0^\infty \frac{1-\cos\omega t}{\omega^2}\, d\omega 
= t \,\frac{2\beta}\pi \int_0^\infty \frac{1-\cos\omega t}{\omega^2t^2}\, d\omega t = \beta\, t\,.
\]
By estimating the sum separately for $\omega_j t \leq 1$ and $\omega_j t>1$ we find that there exists $C>0$ such that 
\[
\Delta \omega \|\alpha_n\|_{\R^{2n}}^2\leq C (t_1-t_0)\,
\qquad\text{for all }n\geq 1 \text{ and }t_0,t_1\in[0,T]\,,
\]
and therefore $\|X_n\|_{\psi_2}$ is bounded by $C \sqrt{t_1-t_0}$. 

By Proposition~2.5.2.(ii) of \cite{Vershynin18} it then follows that for all $p\geq 1$ we have 
\[
\limsup_{n\to\infty} \Expectation |X_n|^p \leq 
C_p \limsup_{n\to\infty} \|X_n\|_{\psi_2}^p
\leq C_p' |t_1-t_0|^{p/2},
\]
for some constants $C_p$, $C_p'$ independent of $n$. This proves the time control that allows us to apply~\cite[p.~450]{GikhmanSkorokhod69}. 

\bigskip
We conclude that the sequence $W_n$ of processes converges in distribution on $C([0,T])$ to $W$. This concludes the proof of the lemma. 
\end{proof}

\printbibliography

\end{document}